\documentclass[10pt,a4paper]{article}

\usepackage{amsmath, amsthm, amssymb, amsfonts}

\usepackage{pstcol}

\usepackage{subfigure}
\usepackage{graphicx,epsfig,color}
\usepackage{scrpage2}
\usepackage{exscale}
\usepackage{pstricks}

\newcommand{\Keywords}[1]{\par\noindent{\small{\em Keywords\/}: #1}}

\theoremstyle{plain}

\newtheorem{theorem}{Theorem}[section]

\newtheorem{lemma}[theorem]{Lemma}

\newtheorem{corollary}[theorem]{Corollary}
\newtheorem{conjecture}[theorem]{Conjecture}
\newtheorem{proposition}[theorem]{Proposition}

\theoremstyle{definition}

\newtheorem{definition}[theorem]{Definition}

\theoremstyle{remark}
\newtheorem{remark}[theorem]{Remark}

\newtheorem*{ak}{Acknowledgements}

\DeclareSymbolFont{AMSb}{U}{msb}{m}{n}
\DeclareMathSymbol{\N}{\mathalpha}{AMSb}{"4E}
\DeclareMathSymbol{\R}{\mathalpha}{AMSb}{"52}
\DeclareMathSymbol{\Z}{\mathalpha}{AMSb}{"5A}
\DeclareMathSymbol{\D}{\mathalpha}{AMSb}{"44}
\DeclareMathSymbol{\s}{\mathalpha}{AMSb}{"53}
\DeclareMathOperator{\diam}{diam}
\newcommand{\M}{M}
\newcommand{\X}{X}
\newcommand{\dx}{d_X}
\newcommand{\spt}{\mathsf{spt}}
\newcommand{\de}{d}
\newcommand{\m}{m}

\newcommand{\sk}{{\mbox{sn}_K}}
\newcommand{\ck}{{\mbox{cn}_K}}
\DeclareMathOperator{\ric}{ric}
\newcommand{\vol}{\mbox{vol}}

\DeclareMathOperator{\supp}{supp}
\begin{document}
\title{Ricci Curvature Bounds for Warped Products}

\author{{Christian Ketterer}\footnote{IAM Bonn, ketterer@iam.uni-bonn.de}}

\maketitle

\begin{abstract}
We prove generalized lower Ricci curvature bounds for warped products over complete Finsler manifolds. 
On the one hand our result covers a theorem of Bacher and Sturm concerning euclidean and spherical cones (\cite{bastco}). On the 
other hand it can be seen in analogy to a result of Bishop and Alexander in the setting of Alexandrov spaces 
with curvature bounded from below (\cite{albi}). For the proof we combine techniques developed in these papers. Because the Finslerian product 
metric can degenerate we regard a warped product as metric measure space that is in general neither a Finsler manifold nor an Alexandrov space again
but a space satisfying a curvature-dimension condition in the sense of Lott-Villani/Sturm.
\Keywords{warped product, curvature-dimension condition, Finsler manifold, Alexandrov space}
\end{abstract}
\maketitle

\tableofcontents
\section{Introduction}
By using displacement convexity of certain entropy functionals on the $L^2$-Wasserstein space 
Sturm in \cite{stugeo1, stugeo2} and Lott-Villani in \cite{lottvillani}
independently introduced a synthetic notion of generalized lower Ricci curvature bound for metric measure spaces $(X,d_X,m_X)$. 
Combined with an upper bound $N\in[1,\infty)$ for the dimension that leads to the so-called 
curvature-dimension condition $CD(K,N)$.
For complete weighted Riemannian manifolds this is equivalent to having bounded $N$-Ricci curvature in the sense of Bakry and Emery. 
The strength of this approach comes from the stability under Gromov-Hausdorff convergence and the numerous corollaries like Brunn-Minkowski inequality, 
Bishop-Gromov volume growth and eigenvalue estimates that now hold in a very general setting.

The other object we deal with is the concept of warped product $B\times_f F$ between two metric spaces $(B,d_B)$ and $(F,d_F)$ and a Lipschitz function 
$f:B\rightarrow [0,\infty)$. It is a generalization of the ordinary cartesian product. Examples are euclidean or spherical cones, 
where the first factor is $[0,\infty)$ or $[0,\pi]$, respectively, and the warping functions are $r$ or $\sin r$.
These spaces appear naturally as tangent cones with respect to other spaces. But also more complicated warped product constructions are important and can be tools 
that allow to prove results beside the warped product itself (see for example \cite{lobaem}).

The aim of this paper is to draw some connections between warped products and the curvature-dimension condition when the underlying spaces are weighted Finsler manifolds. 
More precisely, we want to deduce a curvature-dimension bound for $B\times_f F$ provided suitable conditions for $B$, 
$F$ and $f$. On the one hand, such a result would be interesting because examples that mainly inspired the definition
of curvature-dimension in the sense of Lott-Villani/Sturm came from measured Gromov-Hausdorff limits of Riemannian manifolds. On the other hand, we already have 
stability of $CD$ under tensorization with respect to the cartesian product. 
So the stability under the warped product construction would yield a new class of examples and would be another justification for the new approach. 
The main theorem is
\begin{theorem}\label{MainTheorem}
Let $B$  be a complete, $d$-dimensional space with curvature bounded from below (CBB) by $K$ in the sense of Alexandrov such that $B\backslash\partial B$ is a Riemannian manifold. 
Let $f:B\rightarrow\mathbb{R}_{\geq0}$ be $\mathcal{F}K$-concave and smooth on $B\backslash\partial B$. Assume $\partial B\subseteq f^{-1}(\left\{0\right\})$.
Let $(F,m_F)$ be a weighted, complete Finsler manifold. Let $N\geq 1$ and $K_F\in\mathbb{R}$. 
If $N=1$ and $K_F>0$, we assume that $\diam F\leq \pi/\sqrt{K_F}$. In any case $F$ satisfies $CD((N-1)K_F,N)$ where $K_F\in\mathbb{R}$ such that
\begin{itemize}
\item[1.] If $\partial B=\emptyset$, suppose $K_F\geq Kf^2$.
\item[2.] If $\partial B\neq\emptyset$, suppose $K_F\geq 0$ and $|\nabla f|_p\leq\sqrt{K_F}$ for all $p\in \partial B$.
\end{itemize}
Then the $N$-warped product $B\times^N_f F$ satisfies $CD((N+d-1)K,N+d)$.
\end{theorem}
A continuous function $f: B\rightarrow\mathbb{R}_{\geq 0}$ is said to be $\mathcal{F}K$-concave if 
its restriction to every unit-speed geodesic satisfies $u''+Ku\leq0$ in the barrier sense. 
For a Riemannian manifold and smooth $f$ that is equivalent to $\nabla^2 f(\cdot)\leq-Kf|\cdot|^2$ where $\nabla^2 f$ denotes the Hessian.
The boundary of $B$ is defined in the sense of an Alexandrov space.
The notion of a $N$-warped product is a slight generalisation of the usual warped product where $N\geq 1$ is an additional parameter that describes a dimension bound for $F$ (Definition \ref{nwarped}).
We only consider smooth ingredients, so also the measure $m_F$ is assumed to be smooth (see the remark after Definition \ref{definitionNricci}).

\begin{corollary}
Let $B$  be a complete, $d$-dimensional space with CBB by $K$ such that $B\backslash\partial B$ is a Riemannian manifold. 
Let $f:B\rightarrow\mathbb{R}_{\geq0}$ a function such that it is smooth and satisfies $\nabla^2 f=-Kf$ on $B\backslash\partial B$. Assume $\partial B\subseteq f^{-1}(\left\{0\right\})$. 
Let $(F,m_F)$ be a weighted, complete Finsler manifold. Let $N>1$. 
Then the following statements are equivalent
\begin{itemize}
\item[(i)]  
$(F,m_F)$ satisfies $CD((N-1)K_F,N)$ with $K_F\in\mathbb{R}$ such that
\begin{itemize}
\item[1.] If $\partial B=\emptyset$, suppose $K_F\geq Kf^2$.
\item[2.] If $\partial B\neq\emptyset$, suppose $K_F\geq 0$ and $|\nabla f|_p\leq\sqrt{K_F}$ for all $p\in \partial B$.
\end{itemize}
\item[(ii)] The $N$-warped product $B\times^N_f F$ satisfies $CD((N+d-1)K,N+d)$
\end{itemize}
\end{corollary}
Why could we expect such a relation? To answer this question let us consider first the special case where
$f$ does not vanish and $B$, $F$ and $B \times_f F$ are Riemannian manifolds where $\dim F=n$. The formula for the Ricci tensor 
of warped products is
\begin{equation}\label{equation}
\ric_{B\times_f F}(\tilde{X}_{(p,x)}+\tilde{V}_{(p,x)})=\ric_B(X_p)-n\textstyle{\frac{\nabla^2f(X_p)}{{f(p)}}}+\ric_F(V_x)-\left(\textstyle{\frac{\Delta f(p)}{f(p)}}+(n-1)\textstyle{\frac{|\nabla f|^2_p}{f^2(p)}}\right)|\tilde{V}_{(p,x)}|^2
\end{equation}
(see \cite{on}) where $\tilde{X}$ and $\tilde{V}$ are horizontal and vertical lifts on $B\times_f F$ of vector fields $X$ and $V$ on $B$ and $F$, respectively. For the notion of horizontal and vertical lifts we refer to \cite{on}. 
Then, if $\ric_B\geq (d-1)K$, $f$ is $\mathcal{F}K$-concave, and $|\nabla f|^2 + Kf^2\leq K_F$ everywhere in $B$, we get
\begin{equation}\label{something}
\ric_{F}(v,v)\geq(n-1)K_F|v|^2\, \Longrightarrow\,\ric_{B\times_f F}(\xi+v,\xi+v)\geq (n+d-1)K|\xi+v|^2\,
\end{equation}
for every $v\in T F$ and for every $\xi+v\in TB\times_f F=TB\oplus TF$, respectively. (The condition on $f$ will be explained in Proposition \ref{conditions}.)

But even in the smooth setting one problem still occurs. When we allow the function $f$ to vanish, and that will happen in most of the interesting cases, 
the metric tensor degenerates and the warped product under consideration is no longer a manifold. Especially, we have no
notion for the Ricci tensor at singular points. One strategy to solve this problem could be to cut the singularities and consider only what is left over. 
But in general that space neither will be complete nor strictly intrinsic. 

But the warped product together with its distance function and its volume measure is a complete and strictly intrinsic metric measure space, where the curvature-dimension condition in the sense of Lott-Villani/Sturm can be
defined without problem. So it is convenient to state a result in terms of that condition and use tools from optimal transportation theory to circumvent the problem that comes from the singularities. 
A first step in this direction was done by Bacher and Sturm in \cite{bastco} where they show that the euclidean cone over some Riemannian manifold satisfies $CD(0,N)$ if and only if 
the underlying space satisfies $CD(N-1,N)$.  Our main theorem is an analog of this in the setting of warped products and Finsler manifolds.
 
One can see that
the curvature bound that holds by formula (\ref{equation}) on the regular part passes to the whole space provided the given assumptions are fullfilled. In general that will not be true. 
For example consider $N=1$ and the euclidean cone over $F=\mathbb{R}$. Then equation (\ref{equation}) is still true where $f$ does not degenerate, but the cone does not satisfy any curvature-dimension condition (see \cite{bastco}).
The main part of the proof is to show that the set of singularity points does not affect the optimal transportation of mass and therefore, does not
affect the convexity of any entropy functional on the $L^2$-Wasserstein space (under the given assumptions). 

Although our main theorem only works when the underlying spaces are Riemannian manifolds and weighted Finsler manifolds respectively 
the statement that the transport does not see
any singularities (see Theorem \ref{singularitytransport}) holds even in the setting where one can replace Riemannian and Finsler manifolds with Alexandrov spaces and metric measure spaces respectively and, though 
we are not able to prove it yet, we conjecture an analog of our theorem in this setting (see \ref{conj}). 
We also want to mention that there is a very good compatibility of warped products with
the concept of generalized lower (and upper) bound for curvature in the sense of Alexandrov. If we assume $B$ and $F$ to be Alexandrov spaces with CBB by $K$ and $K_F$ respectively, where
 $\partial B\subset f^{-1}(0)$ and $K_F$ and $f$ as in the theorem then the corresponding warped product has CBB by $K_F$. This result was proved by Alexander and Bishop in \cite{albi} (see also Theorem \ref{albi}).

The organisation of the paper goes as follows. In the next section we give a little introduction to warped products and provide all the results that are important for us. 
We treat the general case of metric measure spaces but then
we concentrate on the more specific situation of our theorem, namely Alexandrov spaces, Riemannian manifolds and Finsler manifolds. 
The needed results of Alexander and Bishop from \cite{albi, albi0} are presented.
We also give the 
definition of the modified Bakry-Emery $N$-Ricci tensor for Finsler manifolds and the definition of curvature-dimension condition for metric measure spaces.
In the third section we prove Theorem \ref{MainTheorem}. There, we first generalize the formula of the Ricci tensor for warped products to the setting of weighted 
Riemannian manifolds and $N$-Ricci tensor and to the setting of Finsler manifolds. Then, we prove the theorem 
that states that under our assumptions the optimal transport does not see any singularities. 
We adopt the idea of Bacher and Sturm where the main difference arises in the fact that singularity sets can have dimension bigger than zero.
After the proof of the main theorem we add a self-contained presentation for the result of the existence of optimal maps in warped products that can not directly be derived from the classical result of McCann but from 
more general consideration of Villani in \cite{viltot}.
\begin{ak}
The author would like to thank Prof. Karl-Theodor Sturm for bringing this topic up to his mind and stimulating discussions and Matthias Erbar and Martin Huesmann for useful comments.
He also wants to thank Prof. Miguel Angel Javaloyes for some important remarks on the non-smoothness of Finslerian warped products, 
and Prof. Shin-ichi Ohta, whose recipe for the proof in the Finslerian case of the main theorem we follow. The author also wants to thank the anonymous referee for her/his comments and remarks
which helped to improve the revised version of this article. 
\end{ak}
\section{Preliminaries}
\subsection{Basic definitions and notations}
Throughout this paper,  $(X,d_X)$ always will denote a complete separable metric space and $m_X$ a locally finite measure  on $(X,\mathcal{B}(X))$ with full support. 
That is,  for all $x\in X$ and all sufficiently small $r>0$ the volume $m_X(B_r(x))$ of balls centered at $x$ is positive and finite. We assume that $X$ has more than one point.

The metric space $(X,d_X)$  is called a \textit{length space} or \textit{intrinsic} if $d_X(x,y)=\inf \mbox{L} (\gamma)$ for all $x,y\in X$, 
where the infimum runs over all curves $\gamma$ in $X$ connecting $x$ and $y$. $(X,d_X)$ is called a \textit{geodesic space} or \textit{strictly intrinsic} 
if every two points $x,y\in X$ are connected by a curve $\gamma$ with $d_X(x,y)=\mbox{L}(\gamma)$.
Distance minimizing curves of constant speed are called \textit{geodesics} or minimizers.
The space of all  geodesics $\gamma:[0,1]\to X$ will be denoted by $\Gamma(X):=\Gamma^{[0,1]}(X)$. $(X,d_X)$ is called \textit{non-branching} if for every tuple $(z,x_0,x_1,x_2)$ of points in $X$ for which $z$ is a midpoint of $x_0$ and $x_1$ as well as of $x_0$ and $x_2$, it follows that $x_1=x_2$.
A triple $(X,d_X,m_X)$ with $d_X$ strictly intrinsic will be called \emph{metric measure space}.

$\mathcal{P}_2(\X,\dx)$ denotes the \textit{$L^2$-Wasserstein space} of probability measures $\mu$ on $(\X,\mathcal{B}(\X))$ with finite second moments which means that $\int_\X\dx^2(x_0,x)d\mu(x)<\infty$
for some (hence all) $x_0\in \X$. The \textit{$L^2$-Wasserstein distance} $\de_{W}(\mu_0,\mu_1)$ between two probability measures
$\mu_0,\mu_1\in\mathcal{P}_2(\X,\dx)$ is defined as
$$\de_{W}(\mu_0,\mu_1)=\inf\left\{\left(\int_{\X\times \X}\dx^2(x,y)\,d\pi(x,y)\right)^{1/2}:
\text{$\pi$ coupling of $\mu_0$ and $\mu_1$}\right\}.$$
Here the infimum ranges over all \textit{couplings} of $\mu_0$ and $\mu_1$, i.e. over all probability measures on $\X\times \X$ with marginals $\mu_0$ and $\mu_1$. Equipped with this metric, $\mathcal{P}_2(\X,\dx)$ is a 
complete separable metric space.
The subspace of $m_X$-absolutely continuous measures is denoted by $\mathcal{P}_2(\X,\dx,m_X)$.

\medskip

\begin{definition}
\begin{itemize}
\item[(i)]
A subset $\Xi\subset\X\times\X$ is called \emph{$\dx^2$-cyclically monotone} if and only if for any $k\in\N$ and for any family $(x_1,y_1),\dots,(x_k,y_k)$ of points in $\Xi$ the inequality
$$\sum^k_{i=1}\dx^2(x_i,y_i)\leq\sum^k_{i=1}\dx^2(x_i,y_{i+1})$$
holds with the convention $y_{k+1}=y_1$.
\item[(ii)] Given probability measures $\mu_0,\mu_1$ on $\X$, a probability measure $\pi$ on $\X\times \X$ is called
\emph{optimal coupling} of them iff $\pi$ has marginals $\mu_0$ and $\mu_1$ and
$$\de^2_{W}(\mu_0,\mu_1)=\int_{\X\times \X}\de_X^2(x,y) \, d\pi(x,y).$$
\item[(iii)] A probability measure $\Pi$ on $\Gamma(\X)$ is called \emph{dynamical optimal transference plan} iff the probability measure $(e_0,e_1)_*\Pi$ on $\X\times \X$ is an optimal coupling of the probability measures
$(e_0)_*\Pi$ and $(e_1)_*\Pi$ on $\X$.
\end{itemize}
\end{definition}
Here and in the sequel $e_t:\Gamma(\X)\to \X$ for $t\in[0,1]$ denotes the evaluation map $\gamma\mapsto \gamma_t$.
Moreover, for each measurable map $f:\X\to \X'$ and each measure $\mu$ on $\X$ the  push forward (or image measure) of $\mu$ under $f$ will be denoted by $f_*\mu$.
\medskip
\begin{lemma} \label{optimaltransport}
\begin{itemize}
\item[(i)] For each pair $\mu_0,\mu_1\in\mathcal{P}_2(\X,\dx)$ there exists an optimal coupling $\pi$.
\item[(ii)]
The support of any  optimal coupling $\pi$  is a $\dx^2$-cyclically monotone set.
\item[(iii)] If $\X$ is geodesic then
for each pair $\mu_0,\mu_1\in\mathcal{P}_2(\X,\dx)$ there exists a dynamical optimal transference plan with given initial and terminal distribution:
$(e_0)_*\Pi=\mu_0$ and $(e_1)_*\Pi=\mu_1$.
\item[(iv)] Given any dynamical optimal transference plan $\Pi$ as above, a geodesic $(\mu_t)_{t\in[0,1]}$ in
$\mathcal{P}_2(\X,\dx)$ connecting $\mu_0$ and $\mu_1$ is given by
$$\mu_t:=(e_t)_*\Pi.$$
\end{itemize}
\end{lemma}
$\rightarrow$ Proofs can be found in \cite{stugeo1},\cite{viltot}.
\\
\\
We assume that the reader is familiar with basics in Alexandrov spaces (\cite{bbi}, \cite{bgp}, \cite{petsem}) and Finsler geometry (\cite{shen}, \cite{ohtafinsler1}, \cite{bcs}).
For a bilinear form $b:V\times V\rightarrow \mathbb{R}$ over some vector space $V$ we will sometimes use the following abbreviation $b(v,v)=:b(v)$.
\subsection{Warped products of metric spaces} 
Let $(B,d_B)$ and $(F,d_F)$ be metric spaces that are complete, locally compact and intrinsic. Let $f: B\rightarrow \mathbb{R}_{\geq0}$ be Lipschitz. 
Let us consider a continuous curve $\gamma=(\alpha,\beta):\left[a,b\right]\rightarrow B\times F$. We define the length of $\gamma$ by
\begin{equation}\label{length}
 \mbox{L}(\gamma):=\sup_T \sum_{i=0}^{n-1}\left(d_B(\alpha(t_i),\alpha(t_{i+1}))^2+f(\alpha(t_{i+1}))^2d_F(\beta(t_i),\beta(t_{i+1}))^2\right)^{\frac{1}{2}}.
\end{equation}
where the supremum is taken with respect to $\left\{(t_i)_{i=0}^n\right\}=:T\subset[a,b]$ with $a=t_0<\dots<t_n=b$.
 We call a curve $\gamma=(\alpha,\beta)$ in $B\times F$ 
admissible if $\alpha$ and $\beta$ are Lipschitz in 
$B$ and $F$ respectively. In that case
\begin{displaymath}
\mbox{L}(\gamma)=\int_0^1\sqrt{v^2(t)+(f\circ \alpha)^2(t) w^2(t)}dt.
\end{displaymath}
where $v$ and $w$ are the metric speeds of $\alpha$ and $\beta$ in $(B,d_B)$ and $(F,d_F)$ respectively. For example 
\begin{displaymath}
v(t)=\lim_{\epsilon\rightarrow 0}\frac{d_B(\alpha(t),\alpha(t+\epsilon))}{\epsilon}.
\end{displaymath}
$\mbox{L}$ is
a length-structure on the class of admissible curves.
For details see \cite{bbi} and \cite{albi0}. 
Then we can define a semi-distance between $(p,x)$ and $(q,y)$ by
\begin{displaymath}
\inf \mbox{L}(\gamma)=: \big|(p,x),(q,y)\big|\in[0,\infty)
\end{displaymath}
where the infimum is taken over all admissible curves $\gamma$ that connect $(p,x)$ and $(q,y)$.
\begin{definition}
The warped product of metric spaces $(B,d_B)$ and $(F,d_F)$ with respect to a locally Lipschitz function $f:B\rightarrow\mathbb{R}_{\geq 0}$ is given by
\begin{eqnarray*}
 \left(C:=B\times F/_{\sim},|\cdot,\cdot|\right)=:B\times_f F
\end{eqnarray*}
where the equivalence relation $\sim$ is given by
\begin{eqnarray*}
 (p,x)\sim(q,y)\Longleftrightarrow \left|(p,x),(q,y)\right|=0
\end{eqnarray*}
and the metric distance is $\textstyle{\left|\left[(p,x)\right],\left[(q,y)\right]\right|:=\left|(p,x),(q,y)\right|}$. 
\end{definition}
\begin{remark}One can see that $$C=(\hat{B}\times_{\hat{f}}F)\dot{\cup}f^{-1}(\left\{0\right\})\mbox{ where }B\backslash f^{-1}(\left\{0\right\})=:\hat{B}\mbox{ and }\hat{f}=f|_{\hat{B}}.$$ 
We will often make use of the notation $\hat{C}:=\hat{B}\times_{\hat{f}} F$.
$B\times_f F$ is intrinsic. Completness and local compactness follow from the corresponding properties
of $B$ and $F$.
It follows that $B\times_f F$ is strictly intrinsic. Especially for
every pair of points we find a geodesic between them. \\
\\
The next two theorems by Alexander and Bishop
describe the behaviour of geodesics.\end{remark}
\begin{theorem}[\mbox{\cite[Theorem 3.1]{albi0}}]\label{fundamentaltheorem}
For a minimizer $\gamma=(\alpha,\beta)$ in $B\times_f F$ with $f>0$ we have
\begin{itemize}
\item[1.] $\beta$ is pregeodesic in $F$ and has speed proportional to $f^{-2}\circ\alpha$. .
\item[2.] $\alpha$ is independent of $F$, except for the length of $\beta$. 
\item[3.] If $\beta$ is non-constant, $\gamma$ has a parametrization proportional to arclength satisfying the energy equation
$
\frac{1}{2}v^2+\frac{1}{2f^2}=E
$
almost everywhere, where $v$ is the speed of $\alpha$ and $E$ is constant.
\end{itemize}
\end{theorem}
%

\begin{theorem}[\mbox{\cite[Theorem 7.3]{albi}}]\label{fundamentaltheorem2}
Let $\gamma=(\alpha,\beta)$ be a minimizer in $B\times_f F$ that intersects $f^{-1}(\left\{0\right\})=:X$.
\begin{itemize}
 \item[1.] If $\gamma$ has an endpoint in $X$, then $\alpha$ is a minimizer in $B$.
\item[2.] $\beta$ is constant on each determinate subinterval.
\item[3.] $\alpha$ is independent of $F$, except for the distance between the endpoint values of $\beta$. The images of the other determinate
subintervals are arbitrary.
\end{itemize}
\end{theorem}
\begin{remark}
 A pregeodesic is a curve, whose length is distance minimizing but not necessarily of constant speed. A determinate subinterval $J$ of definition for $\beta$ is an intervall 
 where $f\circ \alpha$ does not vanish, e.g. $t\in J$ if $f\circ \alpha (t)>0$.
\end{remark}
\begin{definition}
For a metric space $(M,d)$, the \textit{$K$-cone} $(\mbox{Con}_K(M),d_{\mbox{Con}_K})$ is a metric space defined as follows:
\begin{itemize}
\item[$\diamond$]
$\mbox{Con}_K(M):=
\begin{cases} M\times \left[0,\pi/\scriptstyle{\sqrt{K}}\right]/(M\times \left\{0,\pi/\scriptstyle{\sqrt{K}}\right\})&\mbox{ if }K>0\\
						   M\times [0,\infty)/(M\times \left\{0\right\})  &\mbox{ if }K\leq 0
                                                    \end{cases}
$
\item[$\diamond$] For $(x,s),(x',t)\in \mbox{Con}_K(M)$
\begin{align*}&d_{\mbox{Con}_K}((x,s),(x',t))\\
&:=\begin{cases}\ck^{-1}\left(\ck(s)\ck(t)+K\sk(s)\sk(t)\cos\left({d}(x,x')\wedge\pi\right)\right))&\mbox{ if }K\neq 0\\
                                                  \sqrt{s^2+t^2-2st\cos\left({d}(x,x')\wedge\pi\right)}&\mbox{ if }K=0.
                                                 \end{cases}
\end{align*}
\end{itemize}
where $\sk(t)=\frac{1}{\sqrt{K}}\sin(\sqrt{K}t)$ and $\ck(t)=\cos(\sqrt{K}t)$ for $K>0$ 
and $\sk(t)=\frac{1}{\sqrt{-K}}\sinh(\sqrt{-K}t)$ and $\ck(t)=\cosh(\sqrt{-K}t)$ for $K<0$. 
\end{definition}
The $K$-cone with respect to $(M,d)$ is an intrinsic (resp. strictly intrinsic) metric spaces if and only if $(M,d)$ is intrinsic (resp. strictly intrinsic) at distances less than $\pi$
(see \cite[Theorem 3.6.17]{bbi} for $K=0$). 
If $\mbox{diam} M\leq \pi$, the $K$-cone coincides with the warped product $[0,\infty)\times_{\sk} M$.
\subsection{Alexandrov spaces} 
In this section let $B$ and $F$ be finite-dimensional Alexandrov spaces with CBB by $K$ and $K_F$ respectively. 
A continuous function $f: B\rightarrow\mathbb{R}_{\geq 0}$ is said to be $\mathcal{F}K$-concave if 
its restriction to every unit-speed geodesic $\gamma$ satisfies 
$$f\circ\gamma(t)\geq \sigma^{(1-t)}f\circ\gamma(0)+\sigma^{(t)}f\circ\gamma(\theta)\mbox{ for all }t\in[0,\theta]$$
where $\theta=\mbox{L}(\gamma)$. For the definition of $\sigma^{(t)}=\sigma^{(t)}_{1,1}(\theta)$ see the remark directly after Definition \ref{CD}. This is equivalent to that $f\circ\gamma$ is
a sub-solution of the following Dirichlet problem
\begin{align*}
\hspace{12pt}u''&=-K u \hspace{5pt}\mbox{ on }(0,\theta)\\
u(0)=f\circ&\gamma (0),\hspace{9pt}u(\theta)=f\circ\gamma (\theta).
\end{align*}
We assume that $f$ is locally Lipschitz, 
and
\begin{eqnarray}\label{conditions1}
K_F\geq Kf^2(p)&\mbox{ and }& D f_p\leq \sqrt{K_F-Kf^2(p)}\hspace{5mm}\forall p\in B.
\end{eqnarray}
$D f_p$ is the modulus of the gradient of $f$ at $p$ in the sense of Alexandrov geometry (see for example \cite{petsem}). 
In the following we will mainly stay in the setting of Riemannian manifolds, where $Df_p$ can always be replaced by $|\nabla f_p|$. 
\\
\\
Set $X:=(f|{\partial B})^{-1}(0)$. 
In \cite{albi} Alexander and Bishop prove the following the following result.
\begin{proposition}[\mbox{\cite[Proposition 3.1]{albi}}]\label{conditions}
For a $\mathcal{F}K$-concave function $f:B\rightarrow[0,\infty)$ on some Alexandrov space $B$ with CBB by $K$, the condition (\ref{conditions1}) is equivalent to
\begin{itemize}
\item[1.] If $X=\emptyset$, suppose $K_F\geq Kf^2$.
\item[2.] If $X\neq\emptyset$, suppose $K_F\geq 0$ and $D f_p\leq\sqrt{K_F}$ for all $p\in X$.
\end{itemize}
\end{proposition}
In the proof of the previous proposition Alexander and Bishop especially deduce the following result.
\begin{proposition}[\mbox{\cite[Proof of Proposition 3.1]{albi}}]\label{boundary}
Let $f$ be $\mathcal{F}K$-concave and assume it is not identical $0$. Let $B$ be as in Proposition \ref{conditions}. 
Then $f$ is positive on non-boundary points: $f^{-1}\left(\left\{0\right\}\right)\subset\partial B$. Especially $X=f^{-1}(0)$.
\end{proposition}
\begin{proposition}[\mbox{\cite[Proposition 7.2]{albi}}]\label{singularity}
Let $f:B\rightarrow \mathbb{R}_{\geq 0}$ and $B$ as in the previous proposition. Suppose $X=f^{-1}(\left\{0\right\})\neq\emptyset$ and $K_F\geq 0\mbox{ and }Df_p\leq\sqrt{K_F}\mbox{ for all }p\in X.$ Then we have:
Any minimizer in $B\times_f F$ joining two points not in $X$, and intersecting $X$, consists of two horizontal segments whose projections to
$F$ are $\pi/\sqrt{K_{F}}$ apart, joined by a point in $X$.
\end{proposition}
The main theorem of Alexander and Bishop concerning warped products is:
\begin{theorem}[\mbox{\cite[Theorem 1.2]{albi}}]\label{albi}
 Let $B$ and $F$  be complete, finite-dimensional spaces with CBB by $K$ and $K_F$ respectively. 
Let $f:B\rightarrow\mathbb{R}_{\geq0}$ be an $\mathcal{F}K$-concave, 
locally Lipschitz function satisfying the boundary condition ($\dagger$). Set $X=f^{-1}(\left\{0\right\})\subset\partial B$.
\begin{itemize}
 \item[1.] If $X=\emptyset$, suppose $K_F\geq Kf^2$.
\item[2.] If $X\neq\emptyset$, suppose $K_F\geq 0$ and $Df_p\leq\sqrt{K_F}$ for all $p\in X$.
\end{itemize}
Then the warped product $B\times_f F$ has CBB by $K$. 

($\dagger$) If $B^{\dagger}$ is the result of gluing two copies of $B$ on the closure of the set of boundary points where $f$ is 
nonvanishing, and $f^{\dagger}: B^{\dagger}\rightarrow \mathbb{R}_{\geq 0}$ is the tautological extension of $f$, 
then $B^{\dagger}$ has CBB by $K$ and $f^{\dagger}$ is $\mathcal{F}K$-concave.
\end{theorem}
From now on we assume that $\partial B\subset f^{-1}(\left\{0\right\})$, so the boundary condition ($\dagger$) does not play a role and $X= f^{-1}(\left\{0\right\})=\partial B$.
\subsection{Warped products of Riemannian manifolds} 
We assume additionally that $B\backslash f^{-1}\left(\left\{0\right\}\right)=\hat{B}$ and $F$ are Riemannian manifolds with dimension $d$ and $n$ respectively. 
The Riemannian warped product with respect to $\hat{f}=f|_{\hat{B}}$ is defined in the following way:
\begin{displaymath}
 \hat{C}:=\hat{B}\times_{\hat{f}} F:=(\hat{B}\times F, g).
\end{displaymath}
The Riemannian metric $g$ is given by
\begin{displaymath}
 g:=(\pi_B)^*g_B+(f\circ\pi_B)^2(\pi_F)^*g_F
\end{displaymath}
where $g_B$ and $g_F$ are the Riemannian metrics of $B$ and $F$ respectively. It is clear that $g$ is well-defined and Riemannian. 
The length of a Lipschitz-continuous curve $\gamma=(\alpha,\beta)$ in $\hat{C}$ with respect to the metric $g$ is given by
\begin{displaymath}
\mbox{L}(\gamma)=\int_0^1\sqrt{g_B(\dot{\alpha}(t),\dot{\alpha}(t))+f^2\circ \alpha(t) g_F(\dot{\beta}(t),\dot{\beta}(t))}dt
\end{displaymath}
So the Riemannian distance on $\hat{C}$ is defined by
\begin{displaymath}
 |(p,x),(q,y)\hat{|}=\inf \mbox{L}(\gamma)
\end{displaymath}
where the infimum is taken over all Lipschitz curves that are joining $(p,x)$ and $(q,y)$ in $\hat{C}$. 
It is easy to see that the Riemannian warped product $\hat{C}$ as metric space isometrically embeds in the metric space warped product $C$ and the metrics coincide on $C\backslash f^{-1}(\left\{0\right\})$.
The next proposition is the smooth analogue of Theorem \ref{fundamentaltheorem} and can be found in \cite{on}.
\begin{proposition}
$\gamma=(\alpha,\beta)$ is a geodesic in $\hat{B}\times_{\hat{f}}F$ in the sense that $\nabla_{\dot{\gamma}}\dot{\gamma}=0$ if and only if 
\begin{itemize}
 \item[]\hspace{10pt} (1)\hspace{6pt}$\nabla^B_{\dot{\alpha}}\dot{\alpha}=|\dot{\beta}|^2(f\circ\alpha)\nabla f|_\alpha$ in $B$ \hspace{20pt}(2)\hspace{6pt}$\nabla^F_{\dot{\beta}}\dot{\beta}=\frac{-2}{f\circ\alpha}(f\circ\alpha)'\dot{\beta}$ in $F$.
\end{itemize}
\end{proposition}
\begin{remark} $\nabla^B$ and $\nabla^F$ denote the Levi-Civita connections of $B$ and $F$ respectively.
Because of Theorem \ref{fundamentaltheorem} we know that $|\dot{\beta}|^2f^4\circ\alpha=c$ where $c$ is constant. Then (1) becomes
\begin{equation}\label{ODE}
\nabla_{\dot{\alpha}}\dot{\alpha}=-\nabla \textstyle{\frac{c}{2f^2}}|_{\alpha}.
\end{equation}
If $\beta$ is constant, then $c=0$ and $\alpha$ is a geodesic in $B$. 
(\ref{ODE}) also holds for a metric warped product as long as $B$ is Riemannian.
\end{remark}
It is possible to calculate the Ricci tensor of $\hat{B}\times_{\scriptscriptstyle{\hat{f}}} F$ explicitly. Let $\xi+v\in T\hat{C}_{(p,x)}=TB_p\oplus TF_x$ be arbitrary and let $\tilde{X}$ and $\tilde{V}$ be horizontal and vertical lifts 
of vector fields $X$ and $V$ on $\hat{B}$ and $F$ respectively such that $\tilde{X}_{(p,x)}+\tilde{V}_{(p,x)}=\xi+v$. For the notion of horizontal and vertical lifts we refer to \cite{on} and we will use it without further comment. 
Then the formula is
\begin{eqnarray}\label{oneill}
\ric_{\hat{C}}(\xi+v)&=&\ric_{\hat{C}}(\tilde{X}_{(r,x)}+\tilde{V}_{(r,x)})\nonumber\\
&=&\ric_B(X_p)-n\frac{\nabla^2f_p(X_p)}{{f(p)}}+\ric_F(V_x)-\left(\frac{\Delta f(p)}{f(p)}+(n-1)\frac{|\nabla f|^2_p}{f^2(p)}\right)|\tilde{V}_{(p,x)}|^2
\end{eqnarray}
and can be found in \cite[chapter 7, 43 corollary]{on}. Clearly this representation is independent from the vector fields $X$ and $V$.
We remark that $|V|^2=g_F(V,V)$ and $|\tilde{V}|^2=g_{\tilde{C}}(\tilde{V},\tilde{V})$. 
$\nabla^2 f$ denotes the Hessian of $f$ which can be defined as follows. For $v,w\in TM_p$ choose vector fields $X$ and $Y$
such that $X_p=v$ and $Y_p=w$. Then
\begin{eqnarray}\label{lastformula}
 \nabla^2f(v,w)=X_pYf-(\nabla_XY)_pf
\end{eqnarray}
where $\nabla$ denotes here the Levi-Civita-conection of $g$. We declare $\nabla^2f(v)=\nabla^2f(v,v)$. 
Because of (\ref{lastformula}) we will always choose vector fields in the way we did above to do calculation. The results will be independent from this choice. In this smooth setting 
 $\mathcal{F}K$-concavity for a smooth $f$ becomes
$$
\nabla^2 f(v)\leq -fK|v|^2
\mbox{ for any }v\in TB.$$
If we have $\hat{B}$ with $\ric_{\hat{B}}\geq (d-1)K$ and $f$ $\mathcal{F}K$-concave satisfying the condition (\ref{conditions1}) or equivalently the 
conditions in Proposition \ref{conditions}, then
\begin{displaymath}
\ric_{F}(v)\geq(n-1)K_F|v|^2\hspace{5pt}\forall v\in TF\,\Rightarrow\,\ric_{\tilde{C}}(w)\geq (n+d-1)K|\xi+v|^2\hspace{5pt}\forall(\xi+v)\in T\hat{C}.
\end{displaymath}

\subsection{Finsler manifolds}\label{FinslerManifolds}
\begin{definition}\label{finslermanifold}
A $C^{\infty}$-Finsler structure on a $C^{\infty}$-manifold $M$ is a function $\mathcal{F}:TM\rightarrow[0,\infty)$ satisfying the following conditions:
\begin{itemize}
 \item[(1)] (Regularity) $\mathcal{F}$ is $C^{\infty}$ on $TM\backslash 0_M$, where $0_M:M\rightarrow TM$ with $0_M|_p=0\in TM_p$ denotes the zero section of $TM$.
\item[(2)] (Positive homogeneity) For any $v\in TM$ and positive number $\lambda>0$ we have $\mathcal{F}(\lambda v)=\lambda\mathcal{F}(v)$.
\item[(3)] (Strong convexity) Given local coordinates $(x^i,v^i)^n_{i=1}$ on $\pi^{-1}(U)\subset TM$ for $U\subset M$, then
\begin{equation}\label{fins}
(g_{i,j}(v)):=\left(\frac{1}{2}\frac{\partial^2(\mathcal{F}^2)}{\partial v^i\partial v^j}(v)\right)
\end{equation}
is positive-definite at every $v\in \pi^{-1}(U)\backslash 0$.
\end{itemize}
\end{definition}
We call $(g_{i,j})_{1\leq i,j\leq n}$ fundamental tensor and $(M,\mathcal{F})$ a Finsler manifold. $(g_{i,j})_{i,j}$ can be interpreted as Riemannian metric on the vector bundle 
$$\bigcup_{v\in TM\backslash 0_M}TM_{\pi(v)}\rightarrow TM$$
that associates to every $v_p\in TM_p$ again a copy of $TM_p$ itself. An important property of the fundamental tensor for us is its invariance under vertical rescaling:
$$g_{i,j}(v)=g_{i,j}(\lambda v)\mbox{ for every }\lambda> 0.$$
The Finsler structure induces a distance that makes the Finsler manifold a metric space except for the symmetry of the distance. 
Because we only consider symmetric metrics, we additionally assume
\begin{itemize}
 \item[(4)] (Symmetry) $\mathcal{F}(v)=\mathcal{F}(-v)$.
\end{itemize}
The proof of Theorem \ref{MainTheorem} relies on Proposition \ref{bigformula2} where the symmetry of the Finsler metric does not play a role.  
Thus we expect that in the non-symmetric case no further difficulties arise and an analogue of the main result holds true. 

We need notions of curvature for a Finsler manifold $(M,\mathcal{F})$ but we want to avoid a lengthy introduction. So we follow the lines indicated in chapter 6 in \cite{shen}. Fix a unit vector $v\in TM_p$ and extend it to a 
$C^{\infty}$-vector field $V$ on a neighborhood $U$ of $p$ such that $V_p=v$ and every integral curve of $V$ is a geodesic. That is always possible and we call such a vector field geodesic. Because of the strong convexity property the vector field $V$ 
induces a Riemannian structure on $U$ by
$$g^V_p:=\sum_{i,j=1}^{n}(g_{i,j})(V_p)dx^i_p\otimes dx^j_p\mbox{ for all }p\in U.$$
The \textit{flag curvature} $\mathcal{K}(v,w)$ of $v$ and a linearly independent vector $w\in TM_p$ is the sectional curvature of $v$ and $w$ with respect to $g^V$. Similar the \textit{Finsler-Ricci curvature} $\mbox{ric}(v)$ of $v$ is 
the Ricci curvature of $v$ with respect to $g^V$. These definitions are independent of the choice of the geodesic vector field $V$ (see \cite{ohtafinsler1} and \cite{shen}).
\begin{remark}\label{veryimportantremark}
Although we assume the Finsler structure $\mathcal{F}$ to be $C^{\infty}$-smooth (what we will call just smooth) outside the zero section, the lack of regularity at $0_M$ is worse than one would expect. 
Namely 
$\mathcal{F}^2$ is $C^2$ on $TM$ if and only if $\mathcal{F}$ is Riemannian. Otherwise we only get a regularity of order $C^{1+\alpha}$ for some $0<\alpha<1$. 
(For the statement that we have $C^2$ if and only if we are in a smooth Riemannian setting, see Proposition 11.3.3 in \cite{shen}.)
This fact has important consequences for warped products in the setting of Finsler manifolds. 
\end{remark}
Now we want to define a warped product between Finsler manifolds $(B,\mathcal{F}_B)$ and $(F,\mathcal{F}_F)$ with respect to a smooth function $f:B\rightarrow[0,\infty)$ 
and exactly like in the Riemannian case we can define a warped product Finsler structure explicitly on $\hat{B}\times F$ by
\begin{equation}
 \mathcal{F}_{B\times_f F}:=\sqrt{\mathcal{F}^2_{\hat{B}}\circ (\pi_B)_*+(f\circ\pi_B)^2\mathcal{F}^2_F\circ (\pi_F)_*}.\nonumber
\end{equation}
The notion of warped product for Finsler manifolds is already known and studied by several authors (for example see \cite{radu}).
By Remark \ref{veryimportantremark} it is clear that $\mathcal{F}_{B\times_f F}$ is no Finsler structure on $\hat{B}\times F$ in the sense of our definition. 
It cannot be smooth on $T\hat{B}\times 0_F$ unless $F$ is Riemannian and analogously it cannot be smooth on $0_B\times TF$ unless $B$ is Riemannian. 
Especially it is only possible to define the fundamental tensor where $\mathcal{F}_{B\times_f F}$ is smooth.
\begin{remark}
Warped products between Finsler manifolds extend the class that is given by Definition \ref{finslermanifold} and it seems nearby to define and study this larger class of 
geometric objects. This was done by Javaloyes and S\'anchez in \cite{sanchez}. They introduce conic Finsler manifolds which also cover warped products.
\end{remark}
We can also consider the abstract warped product between Finsler manifolds as metric spaces that we introduced before.
Again the two definitions provide the same notion of length and therefore they produce the same complete metric space, that we call again $B\times_f F=C$. 

It will turn out that
curvature bounds for $B$ in the sense of Alexandrov are essential in our proof where we show a curvature-dimension condition for $B\times_f F$. But a general Finsler manifold will not satisfy such a bound.
So it is convenient to assume that at least $(B,\mathcal{F}_B)$ is purely Riemannian with $\mathcal{F}^2_B=g_B$. In this case the fundamental tensor at $v_{(p,x)}$ where $\mathcal{F}_{B\times_f F}$ is smooth becomes 
\begin{eqnarray*}
 g_{i,j}(v_{(p,x)})=\begin{cases}\scriptstyle{(g_B)_{i,j}(p)}\hspace{5pt}&\mbox{ if }\hspace{5pt}1\leq i,j \leq d \\
\scriptstyle{\frac{1}{2}f^2(p) \frac{\partial^2(\mathcal{F}_F^2)}{\partial v^{i-d}\partial v^{j-d}}((\pi_F)_{*}v_{(p,x)})}\hspace{5pt}&\mbox{ if }\hspace{5pt}d+1\leq i,j \leq d+n\\
\scriptstyle{0}& \mbox{otherwise}.\end{cases}
\end{eqnarray*}

\subsection{$N$-Ricci curvature}
\begin{definition}\label{definitionNricci}
Given a complete $n$-dimensional Riemannian manifold $M$ equipped with its Riemannian distance $d_M$ 
and weighted with a smooth measure $dm_M(x)=e^{-\Psi(x)}d\vol_M(x)$ for some smooth function $\Psi:M\to\R$.
Then for each real number $N>n$ the $N$-Ricci tensor is defined as
\begin{align*}\ric^{N,m_M}(v):=\ric^{N,\Psi}(v)&:=
\ric(v)
+ \nabla^2\Psi (v)- \frac1{N-n}\nabla \Psi \otimes \nabla \Psi (v)\\
&=\ric(v)-(N-n)\frac{\nabla^2 e^{-\Psi^{\frac{1}{N-n}}}(v)}{e^{-\Psi^{\frac{1}{N-n}}(p)}}\end{align*}
where $v\in TM_p$. For $N=n$ we define
$$\ric^{N,\Psi}(v):=
\begin{cases}
\ric(v)
+\nabla^2\Psi(v) &\nabla \Psi (v)=0\\
-\infty & \mbox{else}.
\end{cases}$$
For $1\leq N<n$ we define $\ric^{N,\Psi}(v):=-\infty$ for all $v\neq 0$ and $0$ otherwise.
\end{definition}
We switch again to the setting of weighted Finsler manifolds $(M,\mathcal{F}_M,m_M)$. In this context 
the measure $m_M$ is assumed to be smooth. That means, if we consider $M$ in local coordinates, the measure $m_M$ is absolutely continuous with respect to $\mathcal{L}^n$ and the density is a 
smooth and positive function. 
We remark that there is no canonical volume for Finsler manifolds which would allow us to write $m_M$ as a density like in the Riemannian case. 
Motivated by the previous definition Ohta introduced in \cite{ohtafinsler1} for a weighted Finsler manifold
the $N$-Ricci tensor that we define now.
For $v\in TM_p$ choose a geodesic vector field $V$ on $U\ni p$ such that $v=V_p$. Then the geodesic $\eta:(-\epsilon,\epsilon)\rightarrow M$ with $\dot{\eta}(0)=v$ is an integral curve of $V$. Like above that leads to 
 a Riemannian metric
$g^V$ on $U$ and we have the following representation $m_M=e^{-\Psi_V}dvol_{g^V}$ on $U$. Then for $N\geq 1$ the $N$-Ricci tensor at $v$ is defined as
$$\mbox{ric}^{N,m_M}(v):=\ric^{N,\Psi_V}(v).$$
The benefit of this definition will appear shortly after we introduced the curvature-dimension condition.

Of course in the case of our Finslerian warped product this definition makes no sense for $v$ where the Finsler structure is not at least $C^2$. But if the $F$-component of some vector $(\xi,v)$ is non-zero, 
then we will always find a corresponding geodesic vector field with non-zero $F$-components on some open neighborhood that makes it possible to define the $N$-Ricci tensor as above. 

\subsection{The curvature-dimension condition} In this section we give a short survey about the curvature-dimension condition in sense of Lott-Villani/Sturm. (See \cite{stugeo1},\cite{stugeo2},\cite{lottvillani}.)
\begin{definition}\label{CD}
Let $(M,d,m)$ be a metric measure space. Given  $K\in\mathbb{R}$ and $N\in[1,\infty)$, the condition 
$CD(K,N)$ states that for each pair
$\mu_0,\mu_1\in\mathcal{P}_2(M,d,m)$  there exist  an optimal
coupling $q$ of $\mu_0=\rho_0m_M$ and $\mu_1=\rho_1m_M$ and a geodesic 
$\mu_t=\rho_t m_M$ in $\mathcal{P}_2({M},{d}_M,{m}_M)$ connecting them such that
\begin{equation*}
\begin{split}
\int_M\rho_t^{1-1/N'}dm\ge\int_{M\times
M}\left[\tau^{(1-t)}_{K,N'}(d(x_0,x_1))\rho^{-1/N'}_0(x_0)+
\tau^{(t)}_{K,N'}(d(x_0,x_1))\rho^{-1/N'}_1(x_1)\right]dq(x_0,x_1)
\end{split}
\end{equation*}
for all $t\in (0,1)$ and all $N'\geq N$.
\\
\\
In the case $K>0$, the \textit{volume distortion coefficients} $\tau^{(t)}_{K,N}(\cdot)$
for  $t\in (0,1)$  are defined by
$$\tau_{K,N}^{(t)}(\theta)=t^{1/N}\cdot
\left( \sigma_{K,N-1}^{(t)}(\theta)\right)^{1-1/N}\hspace{5pt}\mbox{ where }\sigma_{K,N-1}^{(t)}(\theta)=\frac {\sin\left(\sqrt{ K/{N-1}}\theta t\right)} {\sin\left(\sqrt{K/{N-1}}\theta\right)} $$
if $0\le\theta<\scriptstyle{\sqrt{\frac{N-1}K}\pi}$ and by $\tau_{K,N}^{(t)}(\theta)=\infty$ if $\theta\ge\scriptstyle{\sqrt{\frac{N-1}K}\pi}$.
In the case $K\leq0$ an analogous definition applies with an appropriate replacement of
$\sin\scriptstyle{\left(\sqrt{\frac K{N-1}}\ldots\right)}$.
\end{definition}
The definitions of $CD(K,N)$ in \cite{stugeo1,stugeo2} and \cite{lottvillani} slightly differ. For non-branching spaces, both concepts coincide.
In this case, it suffices to verify (\ref{CD}) for $N'=N$ since this already implies (\ref{CD}) for all $N'\ge N$.
Even more, the condition (\ref{CD}) can be formulated as a pointwise inequality.
\begin{lemma}[\cite{stugeo1},\cite{stugeo2},\cite{lottvillani}]
 A nonbranching metric measure space $(M,d_M,m_M)$ satisfies the curvature-dimension condition
$CD(K,N)$ for  given  numbers $K$ and $N$ if and only if
for each pair
$\mu_0,\mu_1\in\mathcal{P}_2(M,d_M,m_M)$  there exist  a dynamical optimal coupling $\nu$
with initial and terminal distributions $(e_0)_*\nu=\mu_0$, $(e_1)_*\nu=\mu_1$ such that for $\nu$-a.e. $\gamma\in \Gamma(\M)$ and all $t\in(0,1)$
\begin{equation}
\rho_t^{-1/N}(\gamma_t)
\ge\tau^{(1-t)}_{K,N}(\dot\gamma)\cdot\rho^{-1/N}_0(\gamma_0)+
\tau^{(t)}_{K,N}(\dot\gamma)\cdot\rho^{-1/N}_1(\gamma_1)
\end{equation}
where $\dot\gamma:=d_M(\gamma_0,\gamma_1)$ and $\rho_t$ denotes the Radon-Nikodym density of $(e_t)_*\nu$ with respect to $m_M$
\end{lemma}

\begin{theorem}[\cite{stugeo2}, \cite{ohtafinsler1}]\label{theoremsturmohta}
A weighted complete Finsler manifold without boundary $(M,\mathcal{F}_M,m_M)$ satisfies the condition $CD(K,N)$ if and only if
$$\ric^{N,m_M}(v,v)\ge K \mathcal{F}_M^2(v)$$
for all $v\in TM$.
\end{theorem}
\begin{lemma}[Generalized Bonnet-Myers Theorem, \cite{stugeo2}]
Assume that a metric measure space  $(M,d,m)$ satisfies the curvature-dimension condition
$CD(K,N)$ for $N>1$ and $K>0$. Then the diameter of $M$ is bounded by $\pi\textstyle{\sqrt{N-1/K}}$.
\end{lemma}
\begin{theorem}[\cite{rajala2}]
 If $M$ satisfies the curvature-dimension condition $CD(K,N)$, then $M$ satisfies the measure contraction property $(K,N)$-MCP.
\end{theorem}

\paragraph{A New Reference Measure for $B\times_f F$} 
We have to introduce a reference measure on $C$ which reflects the warped product construction. In general we define
\begin{definition}[N-warped product]\label{nwarped}
Let $(B,d_B,m_B)$ and $(F,d_F,m_F)$ be a metric measure spaces. For $N\in [1,\infty)$, the \textit{N-warped product} $(C,|\cdot,\cdot|,m)$ of $B$ and $F$ is 
a metric measure space defined as follows:
\begin{itemize}
\item[-] $C:=B\times_f F=(B\times F/_{\sim},|\cdot,\cdot|)$
\item[-] $dm_C(p,x):=\begin{cases}
                                                            f^N(p) dm_B(p)\otimes dm_F(x)& \mbox{ on }\hat{C}\\
								0 & \mbox{ on } C\backslash\hat{C}.
                                                               \end{cases}$
\end{itemize}
\end{definition}
\begin{remark}
In the setting of $K$-cones we can introduce a measure in the same way. We call the resulting metric measure space a $(K,N)$-cone.
\end{remark}

Let us consider the weighted Riemannian case again. 
On the Riemannian manifold $F$ we have
the reference measure $m_F=e^{-\Psi} d\vol_F$, where $\Psi:F\rightarrow\mathbb{R}$ is a smooth function. 
We do not change the volume on $B$ because
there is already the function $f$ which can be interpreted as weight on $B$. The Riemannian volume of $\hat{C}$ is
\begin{equation*}
 d\vol_{\hat{C}}= f^n d\vol_{\hat{B}}d\vol_F 
\end{equation*}
What is the relation to the measure $m_{C}$ introduced in Definition \ref{nwarped}? For $N\in[1,\infty)$ we have
\begin{eqnarray*}
dm_{\hat{C}}(p,x)=f^N(p)  d\vol_{\hat{B}}(p) e^{-\Psi(x)} d\vol_F(x)=f^{N-n}(p)e^{-\Psi(x)}d\vol_{\hat{C}}(p,x).
\end{eqnarray*}
We define a function $\Phi$ on $\hat{B}$ as follows
\begin{displaymath}
 e^{-\Phi(p)}=f^{N-n}(p)\Longrightarrow \Phi(p)=-(N-n)\log f(p).
\end{displaymath}
So the measure $m_{\hat{C}}$ has the density $e^{-(\Phi+\Psi)}$ with respect to $d\vol_C$.

\section{Proof of the Main Theorem}
\subsection{RicciTensor for Riemannian Warped Products}
Proposition (\ref{bigformula}) is a generalisation of formula (\ref{oneill}) for $N$-warped products for all $N\in [1,\infty)$, where $\hat{B}=B\backslash f^{-1}(0)$ is Riemannian and $(F,\Psi)$ is a weighted Riemannian manifold. 
We remind the reader of the following proposition.
\begin{proposition}
Consider vector fields $X,Y$ on $\hat{B}$ and $V,W$ on $F$ and their horizontal and respectively vertical lifts $\tilde{X}, \tilde{Y}, \tilde{V}, \tilde{W}$. 
$\tilde{\nabla}$ is the Levi-Civita-conection of $\hat{B}\times_{\hat{f}} F$. Then we have on $\hat{B}\times_{\hat{f}} F$
\begin{itemize}
\item[(i)] $\tilde{\nabla}_{\tilde{X}}\tilde{Y}=\widetilde{\nabla^{\hat{B}}_{X}Y}$,
\item[(ii)] $\tilde{\nabla}_{\tilde{X}}\tilde{V}=\tilde{\nabla}_{\tilde{V}}\tilde{X}=\frac{Xf}{f}\circ \pi_B\tilde{V}$,
\item[(iii)] $\tilde{\nabla}_{\tilde{V}}\tilde{W}=-\frac{\left\langle \tilde{V},\tilde{W}\right\rangle}{f}\circ \pi_B\widetilde{\nabla f}+\widetilde{\nabla^F_V W}$.
\end{itemize}
\end{proposition}
\begin{proof} 
The proof is a straightforward calculation or can be found in \cite{on}.
\end{proof}

\begin{proposition}\label{bigformula}
Let $\hat{B}$ be a Riemannian manifold and $(F,\Psi)$ be a weighted Riemannian manifold. Let $N\geq 1$ and $f:\hat{B}\rightarrow (0,\infty)$ is smooth. $\hat{B}\times_f F=\hat{C}$ is the associated 
$N$-warped product of $\hat{B}$, 
$(F,\Psi)$ and $f$. Consider $\xi+v\in T\hat{C}_{(p,x)}=TB_p\oplus TF_x$ and vector fields $X$ and $V$ with $X_p=\xi$ and $V_x=v$ and their horizontal and vertical lifts $\tilde{X}$ and $\tilde{V}$ on $\hat{C}$. Then we have
\begin{itemize}
\item[(i)]  $\ric^{N+d,\Phi+\Psi}_{\hat{C}}(\xi+v)=\ric_{\hat{B}}(\xi)-N\frac{\nabla^2f(\xi)}{f(p)}+\ric^{N,\Psi}_F(v)-\left(\frac{\Delta f(p)}{f(p)}+(N-1)\frac{(\nabla f_p)^2}{f^2(p)}\right)|\tilde{V}_{(p,x)}|^2$.
\end{itemize}
If the reference measure on $F$ is just the Riemannian volume and if we identify $\ric^{N,1}_F$ with $\ric_F$, we especially have
\begin{itemize}
\item[(ii)]  $\ric^{N+d,\Phi}_{\hat{C}}(\xi+v)=\ric_{\hat{B}}(\xi)-N\frac{\nabla^2f(\xi)}{f(p)}+\ric_F(v)-\left(\frac{\Delta f(p)}{f(p)}+(N-1)\frac{(\nabla f(p))^2}{f^2(p)}\right)|\tilde{V}_{(p,x)}|^2$.
\end{itemize}
We remind the reader that $|\tilde{V}_{(p,x)}|^2=f(p)|V_x|^2$.
\end{proposition}

\begin{proof} First we assume $N>n$ where $n$ is the dimension of $F$. We can calculate the $(N+d)$-Ricci curvature explicitly by
leading everything back to formula (\ref{oneill}).
We begin with the second formula. 
That means the reference measure on $F$ is simply the Riemannian volume. 
We have to find espressions for the first and second derivative of $\Phi:\hat{C}\rightarrow \mathbb{R}$. We remind the reader of formula (\ref{lastformula}). 
Set $\bar{W}=\tilde{X}+\tilde{V}$.
\begin{displaymath}
 \nabla^2\Phi(\bar{W})=\nabla^2\Phi(\tilde{X})+\nabla^2\Phi(\tilde{V})+2\underbrace{\nabla^2\Phi(\tilde{X},\tilde{V})}_{=0}
\end{displaymath}
where $\tilde{X}$ and $\tilde{V}$ are lifts of vector fields $X$ and $V$. Further we have
\begin{eqnarray*}
\nabla^2\Phi(\tilde{X})&=&\tilde{X}\tilde{X}\Phi-(\nabla_{\tilde{X}}\tilde{X})\Phi=XX\Phi-(\nabla_XX)\Phi=\nabla^2\Phi(X)\\
&=&-(N-n)\left(X\Big(\frac{Xf}{f}\Big)-\frac{\nabla_XXf}{f}\right)\\
&=&-(N-n)\left(\Big(XXf\cdot f-Xf\cdot Xf\Big)\frac{1}{f^2}-\frac{\nabla_XXf}{f}\right)\\
\nabla^2\Phi(\tilde{V})&=&\tilde{V}\tilde{V}\Phi-(\nabla_{\tilde{V}}\tilde{V})\Phi=-(N-n)\frac{|\nabla f|^2}{f^2}|\tilde{V}|^2\\
\frac{1}{N-n}(\nabla\Phi\otimes\nabla\Phi)(\bar{W})&=&\frac{1}{N-n}(\nabla\Phi\otimes\nabla\Phi)(X)=(N-n)\frac{1}{f^2}(Xf\cdot Xf) 
\end{eqnarray*}
So the $(N+d)-$Ricci curvature becomes
\begin{eqnarray*}
 \ric_{\hat{C}}^{N+d,\Phi}(\bar{W})&=&\ric_{\hat{C}}(\bar{W})\\
&&-(N-n)\bigg(\frac{1}{f^2}\big(XXf\cdot f-Xf\cdot Xf\big)-\frac{\nabla_XXf}{f}
+\frac{|\nabla f|^2}{f^2}|\tilde{V}|^2\bigg)\\
&&-(N-n)\frac{1}{f^2}(Xf\cdot Xf) \\
&=&\ric_{\hat{C}}(\bar{W})\\
&&-(N-n)\bigg(\frac{1}{f^2}XXf\cdot f-\frac{\nabla_XXf}{f}+\frac{|\nabla f|^2}{f^2}|\tilde{V}|^2\bigg)\\
&=&\ric_{\hat{C}}(\bar{W})-(N-n)\frac{\nabla^2f(X)}{f}-(N-n)\frac{|\nabla f|^2}{f^2}|\tilde{V}|^2\\
&=&\ric_{\hat{B}}(X)-N\frac{\nabla^2f(X)}{f}+\ric_F(V)-\left(\frac{\Delta f}{f}+(N-1)\frac{|\nabla f|^2}{f^2}\right)|\tilde{V}|^2
\end{eqnarray*}
Now we change the measure on $F$. There is reference measure $e^{-\Psi} d\vol_F$ on $F$ for a function $\Psi:F\rightarrow\mathbb{R}$.
\begin{eqnarray*}
\nabla^2\Psi(\bar{W})&=&\underbrace{\nabla^2\Psi(\tilde{X})}_{=0}+\nabla^2\Psi(\tilde{V})+2\nabla^2\Psi(\tilde{X},\tilde{V})\\
\nabla^2\Psi(\tilde{V})&=&\tilde{V}\tilde{V}\Psi-(\nabla_{\tilde{V}}\tilde{V})\Psi\\
&=& VV\Psi -(\nabla_VV)\Psi+\langle \nabla f,\nabla\Psi\rangle\frac{1}{f}|V|^2\\
&=&\nabla^2\Psi(V,V)+\langle \nabla f,\nabla\Psi\rangle\frac{1}{f}|V|^2=\nabla^2\Psi(V,V)\\
\nabla^2\Psi(\tilde{V},\tilde{X})&=&\tilde{V}\tilde{X}\Psi-(\nabla_{\tilde{V}}\tilde{X})\Psi\hspace{5pt}=\hspace{5pt}-\frac{1}{f}\tilde{X}  f\cdot\tilde{V}\Psi\\
(\nabla(\Psi+\Phi))\otimes(\nabla(\Psi+\Phi))(\bar{W})&=&(\nabla\Psi\otimes\nabla\Psi)(\tilde{V})+(\nabla\Phi\otimes\nabla\Phi)(\bar{W})-2\frac{N-n}{f}\tilde{X}f\cdot \tilde{V}\Psi
\end{eqnarray*}
The $(N+d)-$Ricci curvature becomes
\begin{eqnarray*}
 \ric^{N+d,\Phi+\Psi}_{\hat{C}}(\bar{W})&=&\ric_{\hat{C}}(\bar{W})+\nabla^2(\Phi+\Psi)(\bar{W})\\
&&-\frac{1}{N-n}\left(\nabla(\Phi+\Psi)\otimes\nabla(\Phi+\Psi)(\bar{W})\right)\\
&=&\ric_{\hat{C}}(\bar{W})+\nabla^2\Phi(\bar{W})-\frac{1}{N-n}(\nabla\Phi\otimes\nabla\Phi)(\bar{W})\\
&&+\nabla^2\Psi(\bar{W})-\frac{1}{N-n}\left((\nabla\Psi\otimes\nabla\Psi)(\tilde{V})-2\frac{N-n}{f}\tilde{X}f\cdot \tilde{V}\Psi\right)\\
&=& \ric_{\hat{C}}^{N+d,\Phi}(\bar{W})+\nabla^2\Psi(\bar{W})\\
&&-\frac{1}{N-n}(\nabla\Psi\otimes\nabla\Psi)(\tilde{V})+2\frac{1}{f}\tilde{X}f\cdot \tilde{V}\Psi\\
&=&\ric_{\hat{B}}(X)-N\frac{\nabla^2f(X)}{f}+\ric_F(V)-\left(\frac{\Delta f}{f}+(N-1)\frac{|\nabla f|^2}{f^2}\right)|\tilde{V}|^2\\
&&+\nabla^2\Psi(V,V)-2\frac{1}{f}\tilde{X}  f\cdot\tilde{V}\Psi\\
&&-\frac{1}{N-n}(\nabla\Psi\otimes\nabla\Psi)(V)+2\frac{1}{f}\tilde{X}f\cdot \tilde{V}\Psi\\
&=&\ric_{\hat{B}}(X)-N\frac{\nabla^2f(X)}{f}+\ric^{N,\Psi}_F(V)-\left(\frac{\Delta f}{f}+(N-1)\frac{|\nabla f|^2}{f^2}\right)|\tilde{V}|^2\\
\end{eqnarray*}
Now we consider the case $N=n$. $\nabla(\Psi+\Phi)(\xi+v)\neq 0$ for $\xi+v\in TC_{(p,x)}$, where $\xi\in TB_p$ and $v\in TF_x$, is
equivalent to $\nabla\Psi(v)\neq 0$ or $\nabla\Phi(\xi)\neq 0$. So by definition the left hand side of our formula 
evaluated at $(\xi,v)$
 is $=-\infty$ if and only if 
the right hand side is $=-\infty$.

If $\nabla(\Psi+\Phi)(\xi+v)=0$, then choose again vector fields $V$ and $X$ with $V_x=v$ and $X_p=\xi$, repeat the above calculation for $N>n$
and evaluate the formula at $(\xi,v)$. The terms with $\frac{1}{N-n}$ disappear. Let $N\rightarrow n$ and get the desired result.

If $N<n$ then $\ric^{N+d,\Psi+\Phi}_{\tilde{C}}=-\infty$ and $\ric^{N,\Psi}_F=-\infty$ by definition. \end{proof}
\subsection{Finsler Case}
Now we treat the case of Finsler manifolds. 

\begin{proposition}\label{bigformula2}
Let $(F,m_F)$ be a weighted Finsler manifold. $N\geq 1$, $f:\hat{B}\rightarrow (0,\infty)$ and $\hat{B}$ are as in the previous proposition. $\hat{B}\times_f F=\hat{C}$ is the associated 
$N$-warped product. Then we have
$$\ric^{N+d,m_C}_{\hat{C}}(\xi+v)=\ric_{\hat{B}}(\xi)-N\frac{\nabla^2f(\xi)}{f(p)}+\ric^{N,m_F}_F(v)-\left(\frac{\Delta f(p)}{f(p)}+(N-1)\frac{(\nabla f_p)^2}{f^2(p)}\right)\mathcal{F}_{\hat{C}}(v)^2$$
where $\xi+v\in T\hat{C}_{(p,x)}$ with $v\neq 0$ (Especially the $N+d$-Ricci tensor of $\hat{C}$ is well-defined at $\xi+v$ because $\mathcal{F}_{\hat{C}}$ is smooth in this direction). 
\end{proposition}
\begin{proof}
Choose $\xi+v\in T\hat{C}_{(p,x)}$ such that $\mathcal{F}_{\hat{C}}(\xi+v)=1$ and with $v\neq 0$ and a unit-speed geodesic $\gamma=(\alpha,\beta):\left[-\epsilon,\epsilon\right]\rightarrow \hat{C}$ with
$\dot{\alpha}(0)=\xi$ and $\dot{\beta}(0)=v$. We set $\gamma(-\epsilon)=(p_0,x_0)$ and $\gamma(\epsilon)=(p_1,x_1)$. We choose $\epsilon$ small such that $\mbox{L}(\gamma)=2\epsilon$ is sufficiently 
far away from the cut radius of the endpoints.
Up to reparametrization $\beta$ is geodesic in $F$ between $x_0$ and $x_1$ by Theorem \ref{fundamentaltheorem}. Let $\bar{\beta}:\left[0,\mbox{L}\right]\rightarrow F$ be the unit-speed 
reparametrization of $\beta$. That means there exists a $s:\left[-\epsilon,\epsilon\right]\rightarrow [0,L]$ such that $\bar{\beta}\circ s=\beta$. $\mbox{L}$ is the length of $\beta$. 
There exists $t_0\in\left[0,\mbox{L}\right]$ such that $\bar{\beta}(t_0)=\beta(0)=x$. We have
$$\dot{\beta}(t)=s'(t)\dot{\bar{\beta}}(s(t))=\mathcal{F}_F(\dot{\beta}(t))\dot{\bar{\beta}}(s(t)).$$
We extend this last observation to the flow of geodesic vector fields $\bar{V}$ and $\bar{W}$.
$$\bar{V}=\nabla d_F(x_0,\cdot)$$ 
is a smooth geodesic vector field on some neighborhood $U$ of $x$. We choose $U$ small enough such that $x_0,x_1\notin U$ and it does not intersect with the cut locus of $x_0$.
Then, if we restrict the image of $\bar{\beta}$ to $U$, it is an integral curve of $\bar{V}$. 
We can define a Riemannian metric $g^{\bar{V}}$ on $U$ with respect to this geodesic vector field and then we can represent the measure $m_F$ with a positive smooth density $\Psi_{\bar{V}}$ with respect to $d\vol_{g^{\bar{V}}}$. 
Additionally, we have
\begin{equation}\label{simple}
 \mathcal{F}_F(\lambda v)=\sqrt{g^{\bar{V}}(\lambda v,\lambda v)} \mbox{  and  } \ric^{N,m_F}(\lambda v)=\ric^{N,\Psi_{\bar{V}}}(\lambda v)\mbox{ for all } \lambda>0
\end{equation}
Consider the vector field 
\begin{displaymath}
\bar{W}=\nabla d_{\hat{C}}((p_0,x_0),\cdot)
\end{displaymath}
restricted to $(\hat{B}\times U)\cap\left\{(p,x):(\pi_F)_*(\bar{W}_{(p,x)})\neq 0\right\}=\hat{U}$ which is open, where $\pi_F$ is the projection from $\hat{C}$ to $F$. Every intergral curve of $\bar{W}$ coincides in $\hat{U}$ with
a unit-speed geodesic from $(p_0,x_0)$ to a point $(\bar{p},\bar{x})\in\hat{U}$. And especially, like we mentioned above, the projection to $F$ of each such geodesic is after reparametrization the geodesic that connects $x_0$ and 
$\bar{x}$ in $F$.
Thus the vertical projections of integral curves of $\bar{W}$ are integral curves of $\bar{V}$ after reparametrization. For an arbitrary integral curve $\gamma=(\alpha,\beta)$ of $\bar{W}$ we do the following explicit computation
\begin{align*}
(\pi_F)_*(\bar{W}_{\gamma(t)})&=(\pi_F)_*(\dot{\gamma}(t))\\
&=(\pi_F)_*(\dot{\alpha}(t)+\dot{\beta}(t))=\dot{\beta}(t)\\
&=\mathcal{F}_F(\dot{\beta}(t))\dot{\bar{\beta}}(s(t))\\
&=\mathcal{F}_F(\dot{\beta}(t))\bar{V}_{\bar{\beta}(s(t))}=\mathcal{F}_F(\dot{\beta}(t))\bar{V}_{\beta(t)}\\
\Rightarrow\hspace{20pt}(\pi_F)_*(\bar{W}_{(r,p)})&=\mathcal{F}_F((\pi_F)_*\bar{W}_{(r,p)})\bar{V}_{p}\hspace{10pt} \forall(r,p)\in\hat{U}
\end{align*}
Thus $\mathcal{F}_F((\pi_F)_*\bar{W})^{-1}\bar{W}=:W$ is $\pi_F$-related to $\bar{V}$. \\
\\
We remember that by definition 
\begin{displaymath}
 \mathcal{F}_{\hat{C}}^2:=\mathcal{F}^2_{\hat{B}}\circ (\pi_B)_*+(f\circ\pi)^2\mathcal{F}^2_F\circ (\pi_F)_*.
\end{displaymath}
The Finslerian $N$-Ricci tensor with respect to $dm_C=d\mbox{vol}_{\hat{B}}\otimes f^N dm_F$ for vectors of $\bar{W}$ is the Riemannian $N$-Ricci tensor of $g^{\bar{W}}$ with respect to $\Psi_{\bar{W}}$. 
The components of $g^{\bar{W}}$ are
$
(g_B)_{i,j}(p)$
if $ 1\leq i,j \leq d$ and
\begin{align*}
g^{\bar{W}}_{i,j}|_{(p,x)}&=\frac{1}{2}\frac{\partial^2(\mathcal{F}_{\hat{C}}^2)}{\partial v^i\partial v^j}(\bar{W}_{(p,x)})\\
&=\frac{1}{2} f^2(p) \frac{\partial^2(\mathcal{F}_F^2)}{\partial v^{i\mbox{-}d}\partial v^{j\mbox{-}d}}((\pi_F)_*\bar{W}_{(p,x)})\\
&=\frac{1}{2} f^2(p) \frac{\partial^2(\mathcal{F}_F^2)}{\partial v^{i\mbox{-}d}\partial v^{j\mbox{-}d}}(\mathcal{F}_F((\pi_F)_*\bar{W}_{(p,x)})\bar{V}_{x})
=f^2(p)g^{\bar{V}}_{i\mbox{-}d,j\mbox{-}d}|_x
\end{align*}
if $d+1\leq i,j \leq n+d$.
The last equality holds because the fundamental tensor is homogenous of degree zero. But then $g^{\bar{W}}|_{(p,x)}=(g_B)|_p+f^2(p)(g^{\bar{V}})|_{x}$ for all $(p,x)\in \hat{U}$
 and we can apply the formula (\ref{oneill}), that especially holds at $\xi+v$. Together with (\ref{simple}) this yields the desired formula for $\xi+v$. If $\mathcal{F}_{\hat{C}}(\xi+v)\neq 1$,
consider the normalized vector, repeat everything and use (\ref{simple}) to get the same result.
\end{proof}
\subsection{Optimal Transportation in Warped Products}
The next theorem is not tied to the context of Finsler manifolds but a purely metric space result.

\begin{theorem}\label{singularitytransport}
Let $(B,d_B)$  be a complete Alexandrov space with CBB by $K$ and let $(F,d_F,m_F)$ be a metric measure space satisfying the $((N-1)K_F,N)$-MCP for $N\geq 1$ and $K_F>0$ and $\mbox{diam}(F)\leq\pi/\sqrt{K_F}$. 
Let $f:B\rightarrow\mathbb{R}_{\geq0}$ be some $\mathcal{F}K$-concave function
such that $X=\partial B= f^{-1}(\left\{0\right\})$, $X\neq\emptyset$ and
\begin{equation*}
D f_p\leq\sqrt{K_F}\mbox{ for all }p\in X.
\end{equation*}
Consider $C=B\times_f F$. 
Let $\Pi$ be an optimal dynamical transference plan in $C$ such that $(e_0)_*\Pi$ is absolutely continuous with respect to 
$m_C$ and $\spt\Pi=\Gamma$. Then the set
\begin{eqnarray*}
\Gamma_X:=\left\{\gamma\in\spt\Pi:\,\exists t\in\left(0,1\right):\gamma(t)\in X\right\}
\end{eqnarray*}
has $\Pi$-measure $0$.
\end{theorem}

\begin{proof} 
We can assume that for all $\gamma\in\Gamma$ there is a $t\in\left(0,1\right)$ such that $\gamma(t)\in X$ and without loss of generality $K_F=1$.
We set $\mu_t=(e_t)_*\Pi$ and $\spt\mu_t=\Omega_t$. $\pi=(e_0,e_1)_*\Pi$ is an optimal plan between $\mu_0$ and $\mu_1$. We assume that $\Omega_0\cap X=\emptyset$. 
For the proof we use the following results of Ohta
\begin{theorem}[\cite{ohtmea}]
If a metric measure space $(M,d,m)$ satisfies the $(K,N)$-MCP for some $K>0$ and $N>1$, then, for any $x\in M$, there exists at most one point $y\in M$ such that $|x,y|=\pi\sqrt{(N-1)/K}$. 
\end{theorem}
\begin{lemma}[\cite{ohtpro}]\label{lemmaohta}
Let $(M,d,m)$ be a metric measure space satisfying the $(K,N)$-MCP for some $K>0$ and $N>1$. If $\mbox{diam} M=|p,q|=\scriptstyle{\pi\sqrt{\frac{N-1}{K}}}$,
then for every point $z\in M$, we have $|p,z|+|z,q|=|p,q|$. In particular there exists a minimal geodesic from $p$ to $q$ passing through $z$. 
\end{lemma}
We want to show that $\mu_0$ is actually concentrated on the graph of some map $\phi:p_1(\Omega_0)\subset B\rightarrow F$ where $p_1: B\times_f^N F\rightarrow B$ is the projection map. 
Then since the measure $\mu_0$ is absolutely continuous with respect to the product measure $f^Nd\mathcal{H}^d_B\otimes dm_F$, 
its total mass has to be zero by Fubinis' theorem and the fact that $m_F$ contains no atoms.
We define $\phi$ as follows. Choose $(p,x)\in\Omega_0$ which is starting point of some transport geodesic $\gamma=(\alpha,\beta)$. If $(p,\tilde{x})\in\Omega_0$, we show that $x=\tilde{x}$. So $\phi$ can be
defined by $p\mapsto x$.
\begin{figure}[h!]
\centering
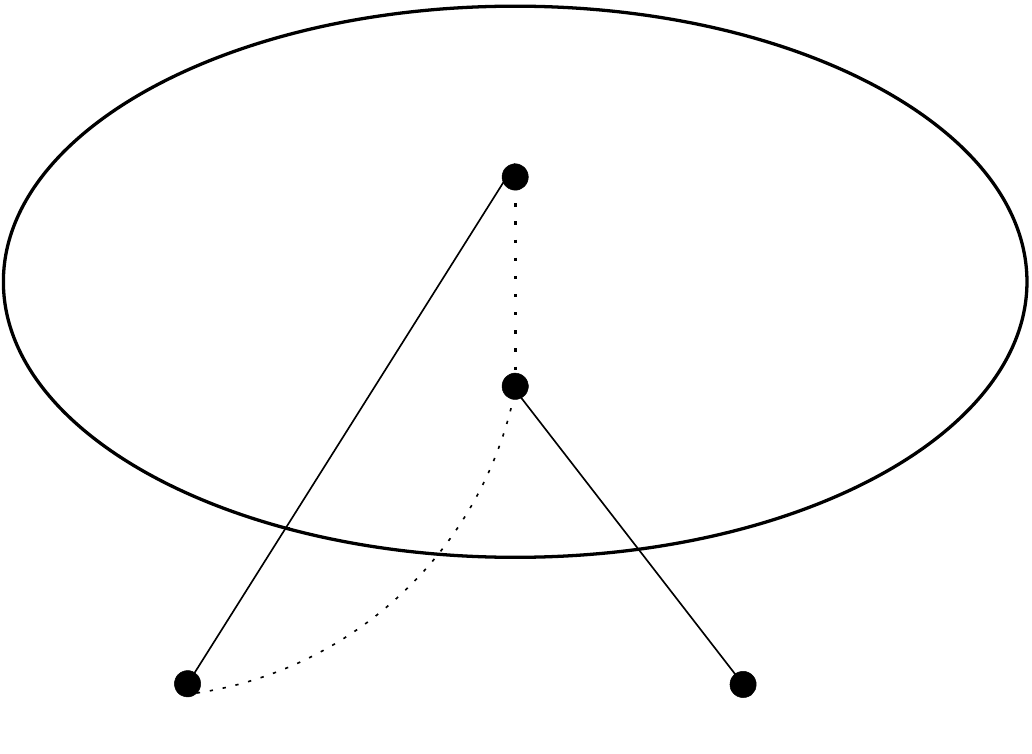
\end{figure}
\\
\\
Let $\gamma,\tilde{\gamma}\in\Gamma$ be transport geodesics starting in $(p,x)$ and $(p,\tilde{x})$, respectively. For the moment we are only concerned with $\gamma=(\alpha,\beta)$. It connects $(p,x)$
and $(q,y)$, and since it passes through $X$, by Proposition \ref{singularity} it decomposes into
$\gamma|_{[0,\tau)}=(\alpha_0,x)$, $\gamma(\tau)=s\in X$ and $\gamma|_{(\tau,1]}=(\alpha_1,y)$ where $x,y\in F$
such that $|x,y|=\pi$. 
\\
\\
We deduce an estimate for $|(p,\tilde{x}),(q,y)|$. 
By Lemma \ref{lemmaohta}  there exists a geodesic from $x$ to $y$ passing throug $\tilde{x}$. So by Theorem 
\ref{fundamentaltheorem} it is enough to consider
$B\times_f \left[0,\pi\right]$ instead of $C$. We have $x=0$ and $y=\pi$. $(\alpha_0,\tilde{x})$ is a minimizer
between $(p,\tilde{x})$ and $s$ and especially $|s,(p,\tilde{x})|=|s,(p,x)|$.
\\
\\
We will essentially use a tool introduced in the proof of Proposition 7.1 in \cite{albi}. There the authors define a 
nonexpanding map $\Psi$ from a section of the constant curvature space $S^3_K$ into $B\times_f\left[0,\pi\right]$. For completeness we repeat its construction:
\\
\\
$B$ is an Alexandrov space, so the following is well-defined. $C_K(\Sigma_s)$ denotes the $K$-cone over 
$\Sigma_s$ where $\Sigma_s$ denotes the space of directions of $s$ in $B$. $s$ is the point where
$\gamma$ intersects $\partial B$. So we can write down the gradient exponential map in $s$
\begin{align*}
 \exp_s:C_K(\Sigma_s)&\rightarrow B\\
(t,\sigma)&\mapsto c_{\sigma}(t)
\end{align*}
where $c_{\sigma}$ denotes the quasi-geodesic that corresponds to $\sigma\in\Sigma_s$. The gradient exponential is a generalisation of the well-known exponential map in Riemannian geometry and
it is non-expanding and isometric along cone radii which correspond to minimizers in $B$. Quasigeodesic were introduce by Alexandrov and studied in detail by Perelman and Petrunin in \cite{pepe}. $\tilde{B}$ denotes 
the doubling of $B$, that is the gluing of two copies of $B$ along their boundaries. By a theorem of Perelman (see \cite{p}) it 
is again an Alexandrov space with the same curvature bound. For $s$ the space of direction $\tilde{\Sigma}_s$ in $\tilde{B}$
is simply the doubling of $\Sigma_s$.
\\
\\
We make the following observations. $\alpha_0\star s \star\alpha_1$ has to be a geodesic in $\tilde{B}$ between $p$ and $q$
where $p$ and $q$ lie in different copies of $B$, respectively. Otherwise there would be a shorter curve $\tilde{\alpha}_0\star \tilde{s} \star\tilde{\alpha}_1$
that would also give a shorter path between $(p,x)$ and $(q,y)$ in $C$. We denote by $\alpha_1^+$ and $\alpha_0^-$ the right hand side and the left hand side tangent vector at $s$, respectively.
\\
\\
By reflection at $\partial B$ we get a another curve
 that is again a geodesic. This curve results from $\alpha_0$ and $\alpha_1$ that were interpreted as curves in the other copy of $B$, respectively. Two cases occur. 
\begin{figure}[h!]
\centering
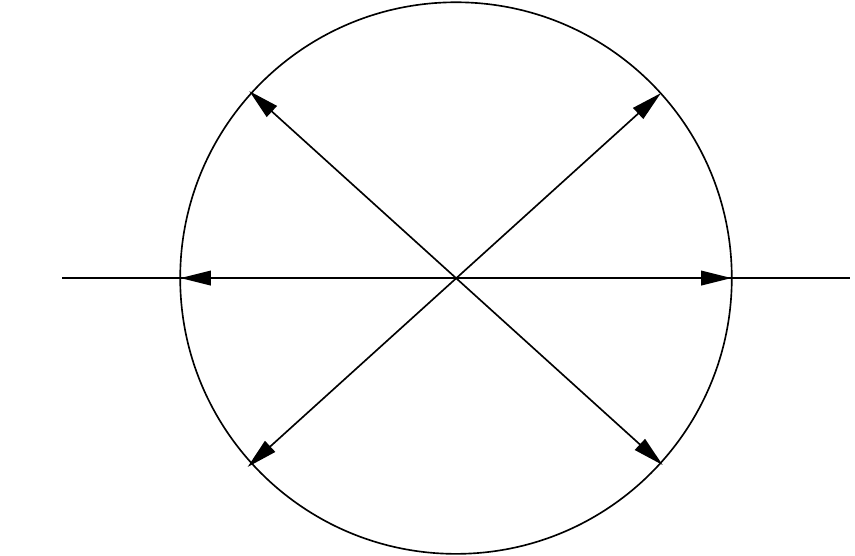
\end{figure}

If $\alpha^+_0(t)\neq\alpha^-_1(t)$ in $\tilde{\Sigma}_s$ then we get two pairs of directions with angle $\pi$. 
The case when $\alpha^+_0(t)=\alpha^-_1(t)$ will be discussed at the end.
Now, in an analog way as one step before, we see that $\tilde{\Sigma}_s$ is a spherical suspension with respect
to each of these pairs and that all 4 directions we consider lie on a geodesic loop $c:[0,2\pi]/_{\left\{0\sim 2\pi\right\}}\rightarrow \tilde{\Sigma}_s$ of length $2\pi$. 
We set $\left\{v_1,v_2\right\}
=\mbox{Im}c\cap \partial{\Sigma_s}$.
Because the second curve was obtained by reflection, clearly we have $|\alpha^+_0,v_1|=|\alpha_0^-,v_1|$ and $|\alpha^+_0,v_2|=|\alpha_0^-,v_2|$ and analogously for $\alpha_1^+$ and $\alpha_1^-$.
So we see that there is an involutive isometry of $\mbox{Im}c$ fixing $\left\{v_1,v_2\right\}$. But then $|v_1,v_2|$ has to be $\pi$. We use a parametrization by arclength such that
$c(0)=v_1$ and $c(\pi)=v_2$ and consider $c|_{[0,\pi]}=c:[0,\pi]\rightarrow \Sigma_s$.
\\
\\
Now consider the space $S^2_K$ of dimension $2$ in $\mathbb{R}^3$ 
and
$S^2_K\cap (\mathbb{R}\times\mathbb{R}_{\geq 0}\times\mathbb{R})=:\hat{S}^2_{K}$. 
We introduce polar coordinates 
\begin{eqnarray*}
\big(\sk(\varphi)\cos(\vartheta),\sk(\varphi)\sin(\vartheta),\ck(\varphi)\big) \mbox{ where } 
\vartheta\in\left[0,\pi\right]\mbox{ and }\varphi\in I_K:=\begin{cases}
                                                     \left[0,\pi/\scriptstyle{\sqrt{K}}\right]&\mbox{if }K>0\\
						     [0,\infty)&\mbox{if }K\leq 0
                                                    \end{cases}
\end{eqnarray*}
and the $K$-cone map 
$$\tilde{\Psi}:\hat{S}^2_{K,}\rightarrow C_K(\Sigma_s)\hspace{5mm}\tilde{\Psi}(\varphi,\vartheta)=(\varphi,c(\vartheta))$$
which is an isometry onto $C_K(\mbox{Im}c\cap\Sigma_s)$.
\\
\\
We consider $\hat{S}^2_{K}\times_{\Phi}\left[0,\pi\right]=:\hat{S}^3_{K}$ where 
$\Phi(\varphi,\vartheta)=\sin\circ d_{\partial\hat{S}^2_K}(\varphi,\vartheta)=\sk\varphi\sin\vartheta$ and 
$\partial\hat{S}^2_K=\left\{(\varphi,\vartheta):\vartheta=0 \mbox{ or }=\pi\right\}\simeq \scriptstyle{\frac{1}{\sqrt{K}}}\textstyle{S^1}$ and define the following map
$$\Psi=\exp_s\circ\tilde{\Psi}\times\mbox{id}_{[0,\pi]}: \hat{S}^2_{K}\times_{\Phi}\left[0,\pi\right]=\hat{S}^3_{K}\rightarrow B\times_f [0,\pi]$$
From the proof of Proposition 7.1 in \cite{albi} we know that 
$\Psi$ is still nonexpanding and an isometry along cone radii which correspond to minimizers in $B$. 
The essential ingredient is $\sk(\varphi)\leq f(\alpha(\varphi))$ for any geodesic $\alpha$ in $B$.
\\
\\
Quite similar as in the case of $K$-cones one can see that the distance on $\hat{S}^2_{K}\times_{\Phi}\left[0,\pi\right]$  is explicetly given by
\begin{eqnarray*}
\ck|(\varphi_0,\vartheta_0,x_0),(\varphi_1,\vartheta_1,x_1)|&=&\ck\varphi_0\ck\varphi_1\\
&&+K\sk\varphi_0\cos\vartheta_0\sk\varphi_1\cos\vartheta_1\nonumber\\
&&+K\sk\varphi_0\sin\vartheta_0\sk\varphi_1\sin\vartheta_1\cos(x_0-x_1)\nonumber\\
&=&\ck\varphi_0\ck\varphi_1\nonumber\\
&&+K\sk\varphi_0\sk\varphi_1\left(\cos\vartheta_0\cos\vartheta_1+\sin\vartheta_0\sin\vartheta_1\cos(x_0-x_1)\right)
\end{eqnarray*}
For $K>0$ we deduce 
the desired estimate
\begin{eqnarray}
 \ck|(p,\tilde{x}),(q,y)|&=&\ck|\Psi((\varphi_0,\vartheta_0,\tilde{x})),\Psi((\varphi_1,\vartheta_1,y))|\nonumber\\
&\geq&\ck|(\varphi_0,\vartheta_0,\tilde{x})),(\varphi_1,\vartheta_1,y)|\nonumber\\
&=&\ck\varphi_0\ck\varphi_1+K\sk\varphi_0\cos\vartheta_0\sk\varphi_1\cos\vartheta_1\nonumber\\
&&+K\sk\varphi_0\sin\vartheta_0\sk\varphi_1\sin\vartheta_1\cos(\tilde{x}-y)\nonumber\\
&\geq&\ck\varphi_0\ck\varphi_1+K\sk\varphi_0\sk\varphi_1(\cos\vartheta_0\cos\vartheta_1-\sin\vartheta_0\sin\vartheta_1)\nonumber\\
&=&\ck\varphi_0\ck\varphi_1+K\sk\varphi_0\sk\varphi_1\cos(\vartheta_0+\vartheta_1)\nonumber\\
&\geq&\ck\varphi_0\ck\varphi_1-K\sk\varphi_0\sk\varphi_1\nonumber\\
&=&\ck(\varphi_0+\varphi_1)\nonumber\\
&=&\ck(|s,(p,\tilde{x})|+|s,(q,y)|)\nonumber\\
&=&\ck(|s,(p,x)|+|s,(q,y)|)=\ck(|(p,x),(q,y)|)\nonumber
\end{eqnarray}
\begin{eqnarray}\label{estimate}
\Longrightarrow \hspace{10pt}|(p,\tilde{x}),(q,y)|&\leq& |(p,x),(q,y)|
\end{eqnarray}
with equality in the second inequality if and only if $|y,\tilde{x}|=\pi$. The case $K\leq 0$ follows in the same way but we have to be aware of reversed inequalities and minus signs that will appear.
We get the same estimate for $(p,x)$ and $(\tilde{q},\tilde{y})$. 
By optimality of the plan we have
\begin{eqnarray}
 |(p,\tilde{x}),(q,y)|^2+|(p,x),(\tilde{q},\tilde{y})|^2&\geq& |(p,x),(q,y)|^2+|(p,\tilde{x}),(\tilde{q},\tilde{y})|^2\nonumber
\end{eqnarray}
and from that we have equality in $(\ref{estimate})$. So we get $|x,\tilde{y}|=\pi$ and $|y,\tilde{x}|=\pi$. But by Ohta's theorem antipodes are unique and thus we get $y=\tilde{y}$ and $x=\tilde{x}$.
\\
\\
The case when $\alpha^+_0(t)=\alpha^-_1(t)$ works as follows. The last identity implies w.l.o.g. $\mbox{Im}\alpha_1\subset\mbox{Im}\alpha_0$. 
We define a map from the $K$-cone into the warped product
\begin{equation*}
\hat{\Psi}:I_K\times_{\sk}\left[0,\pi\right]\rightarrow B\times_f[0,\pi]\mbox{ by }(\varphi,x)\mapsto (\alpha_0(\varphi),x).
\end{equation*}
Again $\hat{\Psi}$ is nonexpanding. By following the lines of Bacher/Sturm in \cite{bastco} we get the same estimate as in (\ref{estimate}).
\end{proof}

\paragraph{Existence of optimal maps}
We have already mentioned that the Finsler structure on $\hat{C}$ is not smooth, or more precisely $\mathcal{F}^2_C$ is $C^1$ but not $C^2$ at any 
$v\in T\hat{B}_p\oplus O_F$. So we cannot apply the classical existence theorem for optimal maps. 
But the special situation of warped products allows to proof the existence of optimal maps by following the lines given in chapter 10 of \cite{viltot}. 
There the cost function comes from a Lagrangian living on a Riemannian manifold. It is easy to see that the Riemannian structure is not so important. 
But the Lagrangian viewpoint fits perfectly well
to our setting if we consider $L:T\hat{C}\rightarrow \mathbb{R}$ with $L(v)=\mathcal{F}_C^2(v)$. The associated action functional is 
$$
\mathcal{A}(\gamma)=\int_0^1\mathcal{F}^2_C(\dot{\gamma}(t))dt
$$
where $\gamma:[0,1]\rightarrow \hat{C}$ is a Lipschitz curve. 
Minimizers of this action functional are just the constant speed geodesics of $\hat{C}$. 
We have the following theorem.
\begin{theorem}\label{optimalmap}
Given $\mu,\nu\in\mathcal{P}^2(\hat{C})$ that are compactly supported and such that $\mu$ is absolutely continuous with respect to $\m_C$. 
Take compact sets $Y\supset\supp\nu$ and $X=\bar{U}$ such that $\supp\mu\subset U$. Then there exists a 
$\frac{1}{2}d^2$-concave function $\phi:X\rightarrow \mathbb{R}_{\geq 0}$ relative to $(X,Y)$ such that the following holds: 
$\pi=(Id_{\hat{C}}, T)_*\mu$ is a unique optimal coupling of $(\mu,\nu)$, where $T:X\rightarrow Y$ is a measurable map and defined $\mu$-almost everywhere 
by $T((p,x))=\gamma^{(p,x)}(1)$ where $\gamma^{(p,x)}$ is a constant-speed geodesic and uniquely determined by $-d\phi_{(p,x)}(\dot{\gamma}^{(p,x)}(0))=\mathcal{F}^2(\dot{\gamma}^{(p,x)}(0))$.
\end{theorem}
For completeness we give a self-contained presentation of the proof from \cite{viltot} in the Appendix where our discussion closely follows \cite{mccann} and \cite{ohtafinsler1}.

\subsection{Proof of the Main Theorem}

\begin{theorem}
Let $B$  be a complete, $d$-dimensional space with CBB by $K$ such that $B\backslash\partial B$ is a Riemannian manifold. 
Let $f:B\rightarrow\mathbb{R}_{\geq0}$ be $\mathcal{F}K$-concave and smooth on $B\backslash\partial B$. Assume $\partial B\subseteq f^{-1}(\left\{0\right\})$.
Let $(F,m_F)$ be a weighted, complete Finsler manifold. Let $N\geq 1$ and $K_F\in\mathbb{R}$. 
If $N=1$ and $K_F>0$, we assume that $\diam F\leq \pi/\sqrt{K_F}$. In any case $F$ satisfies $CD((N-1)K_F,N)$ where $K_F\in\mathbb{R}$ such that
\begin{itemize}
\item[1.] If $\partial B=\emptyset$, suppose $K_F\geq Kf^2$.
\item[2.] If $\partial B\neq\emptyset$, suppose $K_F\geq 0$ and $|\nabla f|_p\leq\sqrt{K_F}$ for all $p\in \partial B$.
\end{itemize}
Then the $N$-warped product $B\times^N_f F$ satisfies $CD((N+d-1)K,N+d)$.
\end{theorem}
\begin{proof}
Let $\partial B\neq\emptyset$. 
For non-constant $f$ we have $K_F>0$. 
Otherwise the warped product is just the ordinary Euclidean product, $K$ has to be nonpositiv and the result is the tensorization property of the $CD$-condition (see \cite{deng}). 
In the case of $N>1$ the curvature-dimension condition for $(F,d_F,m_F)$ implies $((N-1)K_F,N)$-MCP (\cite{rajala2}). 
If $N=1$, then by assumption we have $\diam F\leq \pi/\scriptstyle{\sqrt{K_F}}$.
So in any case Theorem \ref{singularitytransport} yields that positive mass will never transported through the set of singularity points $X$.
So one could think to apply Theorem \ref{theoremsturmohta} to get the result because on $\hat{B}\times^{\scriptscriptstyle{N}}_{\scriptscriptstyle{\hat{f}}} F$ the $N$-Ricci tensor is bounded in the correct way by Proposition
\ref{bigformula} and our assumptions. 

Two problems occur. First, the warped product without its singularity points is not geodesically complete. 
But if we consider some displacement interpolation between bounded and absolutely continuous measures in $\hat{B}\times^{\scriptscriptstyle{N}}_{\scriptscriptstyle{\hat{f}}} F$ then as we have seen 
the transport geodesics do not intersect $X$. So 
by truncation we can find an $\varepsilon$-environment of the singularity set such that the transport takes place in the complement of this environment and the exceptional mass can be chosen arbitrarily small. 
Then in the case where $F$ is Riemannian the calculus 
that was introduced in \cite{CMS} is available like in the complete setting and one gets the convexity of the Jacobian of the optimal map 
along the transport geodesics which leads to the curvature-dimension condition (see also \cite{bastco}). When $\partial B=\emptyset$ this step is redundant because no singularity points appear.

Second, if $F$ is Finslerian, the warped product structure is not
smooth on $T\hat{B}\times O_F$. So we cannot follow the lines of \cite{ohtafinsler1} as we did with \cite{CMS} in the Riemannian case. 
But we know, if $\gamma=(\alpha,\beta)$ is a geodesic in 
$\hat{B}\times^{\scriptscriptstyle{N}}_{\scriptscriptstyle{\hat{f}}} F$ then by Theorem \ref{fundamentaltheorem} $\beta$ is a pre-geodesic. So either $\beta$ is constant and $\alpha$ is a geodesic in $B$, or there exists
a strictly monotone reparametrization $s$ such that $\bar{\beta}=\beta\circ s$ is a constant speed geodesic in $F$. 
We use this fact to circumvent the problem that comes from the non-smoothness. The idea is to split the initial measure of some optimal mass transportation 
in $\hat{B}\times^{\scriptscriptstyle{N}}_{\scriptscriptstyle{\hat{f}}} F$ in
two disjoint parts that will follow one of these two kinds of geodesics either. To do so we need that a point $(p,x)\in\supp\mu_0$ already determines the transport geodesic that starts in $(p,x)$ uniquely. 
But this follows from the existence of an optimal map.

So we proceed as follows. Let $\mu_0$ and 
$\mu_1$ be absolutely continuous probability measure in $\hat{C}$. We assume w.l.o.g. that $\mu_0$ and $\mu_1$ are compactly supported. 
Otherwise, we have to choose compact exhaustions of $\hat{C}\times\hat{C}$ and to consider the restriction of the plan to these compact sets. 
For this we also refer to \cite[Lemma 3.1]{convexfunctionals}.
By Theorem \ref{optimalmap} there is an unique optimal map $T:X\rightarrow Y$ between $\mu_0$ and $\mu_1$. 
So the unique optimal plan is given by $(\mbox{id}, T)_*\mu_0=\pi$ and the associated optimal dynamical plan is given by
$\gamma_*\mu_0=\Pi$ where $\gamma:\supp\mu_0 \rightarrow \mathcal{G}(\hat{C})$ with $(p,x)\mapsto\gamma^{(p,x)}$. 
The geodesic in $\mathcal{P}^2(\hat{C})$ with respect to $L^2$-Wasserstein distance is given by $\mu_t=(\gamma^{(p,x)}_t)_*\mu_0$.
We have
\begin{align*}
\supp\Pi=\underbrace{\big\{\gamma:\dot{\gamma}\in T\hat{B}\times 0_F\big\}}_{=:\Gamma_a}\dot{\cup}\underbrace{\big\{\gamma:\dot{\gamma}\in T\hat{B}\times TF\backslash 0_F\big\}}_{=:\Gamma_b}.
\end{align*}
We set $\Pi(\Gamma_a)^{-1}\Pi|_{\Gamma_a}=:\Pi_a$ and $\Pi(\Gamma_b)^{-1}\Pi|_{\Gamma_b}={\Pi}_b$ that are again optimal dynamical plans. The corresponding $L^2$-Wasserstein geodesics are
$(e_t)_*{\Pi}_a=\mu_{a,t}$ and $(e_t)_*{\Pi}_b=\mu_{b,t}$. They are absolutely continuous with densities $\rho_{a,t}$ and $\rho_{b,t}$ and have disjoint support for any $t\in[0,1]$ 
because of the optimal map and since $\hat{C}$ is non-branching (see \cite[Lemma 2.6]{bast}). 
We have for any $t\in[0,1]$
\begin{align*}
\rho_td\m_C=\mu_t=\Pi({\Gamma}_a){\mu}_{a,t}+\Pi({\Gamma}_b){\mu}_{b,t}=\Pi({\Gamma}_a){\rho}_{a,t}d\m_C+\Pi({\Gamma}_b){\rho}_{b,t}d\m_C.
\end{align*}
So the R\'enyi entropy functional from Definition \ref{CD} splits for any $t\in[0,1]$
\begin{displaymath}
\int_M\rho_t^{1-1/N'}dm_C=\Pi({\Gamma}_a)^{1-1/N'}\!\!\!\int_M\!\!{\rho}_{a,t}^{1-1/N'}dm_C+\Pi({\Gamma}_b)^{1-1/N'}\!\!\!\int_M\!\!{\rho}_{b,t}^{1-1/N'}dm_C
\end{displaymath}
for any $N'\geq N$. So it suffices to show displacement convexity along ${\Pi}_a$ and ${\Pi}_b$ separately.

We begin with $\Pi_a$. We can approximate $\Pi_a$ in $L^2$-Wasserstein distance arbitrarily close by 
$$
\frac{1}{n}\sum_{i=1}^n\Pi_{a,B}^i\otimes\nu_i
$$
where $\Pi_{a,B}^i$ are geometric optimal transference plans in $(B,d_B)$ and $\nu_i$ are disjoint absolutely continuous propability measures in $F$. So it suffices to show displacement convexity along 
$\Pi_{a,B}^i$. But since $B$ has CBB by $K$ and $f$ is $\mathcal{F}K$-concave, $(B,d_B,f^{\scriptscriptstyle{N}}d\vol_B)$ satisfies $CD((N+d-1)K,N+d)$ (see \cite[Theorem 1.7]{stugeo2}) and the desired convexity in $\Pi_{a,B}$ follows at once.

Now consider $\Pi_b$. 
We know a priori that the transport geodesics only follow smooth directions of the Finslerian warped product structure. 
So we can consider $\mathcal{F}^2_C$ restricted to $T\hat{B}\times TF\backslash 0_F$.
We get the exponential map on $T\hat{B}\times TF\backslash 0_F$ and we also can define the Legendre transformation, that yields gradient vector fields. 
Especially, if we consider an optimal transport that follows only smooth direction, the techniques from \cite{ohtafinsler1} can be applied. 
Thus there exists an optimal map $T_b$ of the form $T_b((p,x))=\exp(-\nabla\phi_{(p,x)})$ for some $c$-concave function $\phi$.
To make this more precise we can consider the complement of an $\epsilon$-neighborhood $\mathcal{U}_{\epsilon}$ of $T\hat{B}\times 0_F$ 
and restrict the initial measure $\mu_{b,0}$ of $\Pi_b$ to the set $$U_{\epsilon}=\left\{(p,x)\in\supp\mu_{b,0}:\dot{\gamma}^{(p,x)}(0)\notin\mathcal{U}_{\epsilon}\right\}.$$
$U_{\epsilon}$ is measurable because the mapping $(p,x)\in\mapsto \dot{\gamma}^{(p,x)}$ is measurable.
Again the exceptional mass can be chosen arbitrarily small.
The optimal map $T$, which has been derived in Theorem \ref{optimalmap}, restricted to $U_{\epsilon}$ has to coincide with $T_b$ because 
optimality is stable under restriction and because of uniqueness of optimal maps.
Especially we can deduce $\mu_{b,t}=(T_t)_*\mu_{b,0}$ where $T_t((p,x))=\exp(-t\nabla\phi_{(p,x)})$. 
Again by results from \cite{ohtafinsler1} we know $\phi$ is second order differentiable at least on $U_{\epsilon}$. 
Hence the Jacobian of $T_t$ exists and satisfies because of Proposition 
\ref{bigformula2} and our assumptions the correct convexity condition. Finally one can follow the lines of section 8 in \cite{ohtafinsler1}. \end{proof}
\begin{corollary}\label{col1}
Let $B$  be a complete, $d$-dimensional space with CBB by $K$ 
that is a Riemannian manifold.
Let $f:B\rightarrow\mathbb{R}_{\geq0}$ be $\mathcal{F}K$-concave and smooth.
Assume $\emptyset\neq\partial B\subseteq f^{-1}(\left\{0\right\})$. 
Let $(F,m_F)$ be a weighted, complete Finsler manifold. Let $N>1$. 
Then the following statements are equivalent
\begin{itemize}
\item[(i)]  
$(F,m_F)$ satisfies $CD((N-1)K_F,N)$ with $K_F\geq 0$ and \begin{itemize}
\item[]$|\nabla f|_p\leq\sqrt{K_F}\mbox{ for all }p\in \partial B$.
\end{itemize}
\item[(ii)] The $N$-warped product $B\times^N_f F$ satisfies $CD((N+d-1)K,N+d)$
\end{itemize}
\end{corollary}
\begin{proof}
Only one direction is left.
Assume the $N$-warped product $B\times^N_f F$ satisfies $CD((N+d-1)K,N+d)$. 
Proposition \ref{bigformula2} yields that
\begin{eqnarray}
(N+d-1)K\mathcal{F}_{\hat{B}\times_{\hat{f}}F}^2(\tilde{V})&\leq& 
\ric^{N,m_F}_F(V)-\left(\textstyle{\frac{\Delta f(p)}{f(p)}+(N-1)\frac{|\nabla f|^2_p}{f^2(p)}}\right)\mathcal{F}_{\hat{B}\times_{\hat{f}}F}^2(\tilde{V})\label{second1}
\end{eqnarray}
where $V\in TF_x$ is arbitrary and $\tilde{V}\in T\hat{C}_{(p,x)}$ such that $(\pi_F)_*\tilde{V}=V$. The last inequality is equivalent to
\begin{eqnarray}
(N+d-1)Kf^2(p)\mathcal{F}_{F}^2(V)&\leq& 
\ric^{N,m_F}_F(V)-\left(\Delta f(p)f(p)+(N-1)|\nabla f|^2_p\right)\mathcal{F}_{F}^2(V).
\end{eqnarray}
Now we can choose $p\in B$ independent from $V$ and thus, we let $p$ tend to the non-empty boundary of $B$. Then $\Delta f(p)f(p)$ tends to $0$ because $\Delta f$ is smooth on $B$ (included the boundary) and we get
\begin{eqnarray}\label{est}
(N-1)|\nabla f|^2_p\mathcal{F}_{F}^2({V})\leq \ric^{N,m_F}_F(V)
\end{eqnarray}
for all $p\in \partial B$ and all $V\in TF$. This inequality implies that $(N-1)|\nabla f|^2$ is bounded from above on $\partial B$
by $\textstyle{\mathcal{F}_{F}^2({V})^{-1}}\ric^{N,m_F}_F(V)$ for arbitrary $V\in TF$. So we can set $\sup_{p\in\partial B} |\nabla f|^2_p=K_F<\infty$.
For any $\epsilon>0$ we find $p\in\partial B$ such that $(N-1)|\nabla f|^2_p>(N-1)K_F-\epsilon$. Then we get from (\ref{est})
\begin{eqnarray}
\left((N-1)K_F-\epsilon\right)\mathcal{F}_{F}^2({V})\leq \ric^{N,m_F}_F(V) \nonumber
\end{eqnarray}
and since $\epsilon>0$ is arbitrary we get the desired curvature bound.
\end{proof}

\begin{corollary}
Let $B$  be a complete, $d$-dimensional space with CBB by $K$ such that $B\backslash\partial B$ is a Riemannian manifold. 
Let $f:B\rightarrow\mathbb{R}_{\geq0}$ a function such that it is smooth and satisfies $\nabla^2 f=-Kf$ on $B\backslash\partial B$. Assume $\partial B\subseteq f^{-1}(\left\{0\right\})$. 
Let $(F,m_F)$ be a weighted, complete Finsler manifold. Let $N>1$. 
Then the following statements are equivalent
\begin{itemize}
\item[(i)]  
$(F,m_F)$ satisfies $CD((N-1)K_F,N)$ with $K_F\in\mathbb{R}$ such that
\begin{itemize}
\item[1.] If $\partial B=\emptyset$, suppose $K_F\geq Kf^2$.
\item[2.] If $\partial B\neq\emptyset$, suppose $K_F\geq 0$ and $|\nabla f|_p\leq\sqrt{K_F}$ for all $p\in \partial B$.
\end{itemize}
\item[(ii)] The $N$-warped product $B\times^N_f F$ satisfies $CD((N+d-1)K,N+d)$
\end{itemize}
\end{corollary}
\begin{proof}
Assume the $N$-warped product $B\times^N_f F$ satisfies $CD((N+d-1)K,N+d)$. Like in the proof of previous corollary we can deduce (\ref{second1}). Now we have $\Delta f=-dKf$ on $B\backslash\partial B$ and
we can deduce
\begin{eqnarray*}
 \ric^{N,m_F}_F(V)\geq(N-1)\left(K f^2(p)+|\nabla f|^2_p\right)\mathcal{F}_{F}^2({V})
\end{eqnarray*}
for all $p\in B\backslash \partial B$ and all $V\in TF$. Like in the proof of the previous corollary this inequality implies that $|\nabla f|^2+Kf^2$ is bounded on $B\backslash \partial B$. 
So we can set $\sup_{p\in B\backslash \partial B}|\nabla f|_p^2+Kf^2(p) =:K_F$. (Since $f$ is $\mathcal{F}K$-concave, $|\nabla f|_p^2+Kf^2(p)$ is actually constant on $B$ (see for example \cite{albi}).) This yields 
\begin{eqnarray}
K_F\geq Kf^2(p)&
\hspace{5mm}\forall p\in B.\nonumber
\end{eqnarray}
Then by Proposition \ref{conditions} this is equivalent to the conditions 1. and 2. in the theorem and as in Corollary \ref{col1} the $N$-Ricci tensor of $F$ is bounded by $K_F$.
\end{proof}

\begin{remark}
Like in the theorem of Alexander and Bishop our result can be extended to the case where $B$ satisfies a suitable boundary condition. 

\vspace{3pt}
$(\dagger)$: \textit{$B^{\dagger}$ is the result of gluing two copies of $B$ on the closure of the set of boundary points where $f$ is 
nonvanishing, and $f^{\dagger}: B^{\dagger}\rightarrow \mathbb{R}_{\geq 0}$ is the tautological extension of $f$.
Assume $B^{\dagger}$ has CBB by $K$
and $f^{\dagger}$ is $\mathcal{F}K$-concave.}

\vspace{3pt}
The proof of the main theorem in this situation is exactly the same since $(\dagger)$ implies that 
the warped product $C\backslash \partial C$ is strictly intrinsic.
We do not go into details and refer to \cite{albi}.
\end{remark}
\begin{remark}
If $B=[0,\pi/\sqrt{K}]$ and $f=\sk$ (with appropriate interpretation if $K\leq 0$) and if $\diam F\leq \pi$, the associated warped products are $K$-cones. 
If $F$ is a Riemannian manifold in this setting we get the theorem of Sturm and Bacher from \cite{bastco}. However, if $F$ is Finslerian, the result is new.
\end{remark}
\begin{corollary}
For any real number $N>1$, $CD(N-1,N)$ for a weighted Finsler
manifold is equivalent to $CD(K\cdot N,N + 1)$ for the associated $(K,N)$-cone.
\end{corollary}
\begin{remark}
Theorem \ref{singularitytransport} is true when $B$ is an Alexandrov space and $F$ some general metric measure space. So it is reasonable to assume that our main result also could hold in a non-smooth context and we conjecture the
following
\end{remark}
\begin{conjecture}\label{conj}
Let $(B, d_B)$  be a complete Alexandrov space with $\dim_B=d$ and let $(F,d_F,m_F)$ be a metric measure space. Let $f:B\rightarrow\mathbb{R}_{\geq0}$ be some continuous function
such that $\partial B\subset f^{-1}(\left\{0\right\})$.
Assume that
$(F,m_F)$ satisfies $CD((N-1)K_F,N)$ 
and $f$ is $\mathcal{F}K$-concave such that
\begin{itemize}
\item[1.] If $\partial B=\emptyset$, suppose $K_F\geq Kf^2$.
\item[2.] If $\partial B \neq\emptyset$, suppose $K_F\geq 0$ and $Df_p\leq\sqrt{K_F}$ for all $p\in X$.
\end{itemize}
Then the $N$-warped product $B\times^N_f F$ satisfies $CD((N+d-1)K,N+d)$
\end{conjecture}
\begin{remark}
In \cite{bastco} there is an example where the euclidean cone over some Riemannian manifold 
with $CD(N-1,N)$ produces a metric measure space satisfying $CD(0,N+1)$ but that is not an Alexandrov space with curvature bounded from below. 
They consider $F=\scriptstyle{\frac{1}{\sqrt{3}}} S^2\times \scriptstyle{\frac{1}{\sqrt{3}}} S^2$ which satisfies $CD(3,4)$ but has sectional curvature $0$ for planes spanned by vectors that lie in different spheres. 
Then the sectional curvature bound for the cone explodes when one gets nearer and nearer to the apex.
For general warped products the same phenomenon occurs what can 
be seen at once from the formula of sectional curvature for warped products. Choose any closed $n$-dimensional Riemannian manifold
with Ricci curvature bounded from below by $(n-1)K_F$ and with sectional curvature $K_F(V_x,W_x)=0$ for some vectors $V_x$ and $W_x$ in $TF|_x$
(for example choose $\lambda S^m\times \lambda S^m$ where $m+m=n$ and $\lambda$ is an appropriate scaling factor, that produces the Ricci curvature bound $(n-1)K_F$). 
Let $B$ be a Riemannian manifold with boundary and sectional curvature bigger than $K\in\mathbb{R}$ and $f$ is 
$\mathcal{F}K$-concave and satisfies the assumption of the theorem. (for example choose $B$ as the upper
hemisphere of $S^d$ and $f$ as the first nontrivial eigenfunction of the laplacian of this sphere. Especially $f$ vanishes at the boundary of $B$ and $|\nabla f|_{\partial B}=1$. See \cite{chaeig}). The
sectional curvature of the plane $\Pi_{(p,x)}$ spanned by vectors $(X_p,V_x),(Y_p,W_x)$ in $T(B\times_f F)_{(p,x)}$ is
\begin{eqnarray*}
K(\Pi_{(p,x)})
&=&K_B(X_p,Y_p)|X_p|^2|Y_p|^2-f(p)\left[|W_x|^2\nabla^2f(X_p,X_p)+|V_x|^2\nabla^2f(Y_p,Y_p)\right]\\
&&+\frac{1}{f^2(p)}\left[K_F(V_x,W_x)-|\nabla f_p|^2\right]|\tilde{V}_x|^2|\tilde{W}_x|^2\\
&=&K_B(X_p,Y_p)|X_p|^2|Y_p|^2-f(p)\left[|W_x|^2\nabla^2f(X_p,X_p)+|V_x|^2\nabla^2f(Y_p,Y_p)\right]\\
&&-\frac{1}{f^2(p)}|\nabla f_p|^2|\tilde{V}_x|^2|\tilde{W}_x|^2.
\end{eqnarray*}
Hence
the sectional curvature of planes $\Pi_{(p_n,x)}\subset T(\hat{B}\times^{\scriptscriptstyle{N}}_{\scriptscriptstyle{\hat{f}}} F)_{(p_n,x)}$ as above explodes to $-\infty$ if we choose a sequences $(X_{p_n})$ and $(Y_{p_n})$ such that $p_n$ tends to vanishing points of $f$. On the other hand the Ricci curvature is still bounded
by $0$ by formula (\ref{oneill}). Especially there is also no upper bound for the sectional curvature. 
\end{remark}
Another application of Theorem \ref{singularitytransport} is the following corollary which modifies a theorem by Lott that was proven in \cite{lobaem}.
\begin{corollary}\label{cor}
 Let $(B,g)$ be a compact, $n$-dimensional Riemannian manifold with distance function $d_B$ and with CBB by $K$. Let $f: B\rightarrow \mathbb{R}_{\geq 0}$ 
be a smooth and $\mathcal{F}K$-concave function. Let $N\in\mathbb{N}$ such that $N\geq n$ and set $q=N-n$.
Assume
\begin{align}\label{assumption}
|\nabla f|^2_p\leq K_F\hspace{0.3cm}\forall p\in\partial\supp f.
\end{align}
Then $(\supp f,d_B,f^q d\vol_B)$ is the measured Gromov-Hausdorff limit of a sequence of compact geodesic spaces $(M_i,d_i)$ of Hausdorff dimension $N$ satisfying $CD((N-1)K,N)$.
\end{corollary}
\begin{proof} 
Consider the $q$-warped product $M_i=B\times_{g_i}\textstyle{\frac{i}{\sqrt{K_F}}}S^q$ where $g_i=\frac{1}{i}f$. 
The assumption (\ref{assumption}) implies 
$$
|\nabla g_i|^2_p\leq \frac{K_F}{i^2}\hspace{0.3cm}\forall p\in\partial\supp f.
$$
Then by our main theorem $M_i$ satisfies $CD((N-1)K, N)$ for any $i$ and $(M_i)_i$ converges to $(\supp f,d_B,f^q d\vol_B)$ in measured Gromov-Hausdorff sense for $i\rightarrow 0$. 
\end{proof}

\begin{remark}
In (\ref{bigformula}) and (\ref{bigformula2}) there appears
$
\ric_{\hat{B}}-\textstyle{\frac{N}{f}\nabla^2f}
$
which is actually the $N+d$-Ricci tensor for the weighted Riemannian manifold $(B,g_B,f^{\scriptscriptstyle{N}}d\vol_B)$ (compare with Definition \ref{definitionNricci}). 
So one could think that we could weaken the assumption on $B$ and $f$ from bounds on sectional curvature 
and $\mathcal{F}K$-concavity respectively to a curvature-dimension condition for $(B,g_B,f^{\scriptscriptstyle{N}}d\vol_B)$. 
On the other hand, the proof of Theorem \ref{singularitytransport} 
needs the curvature bound $K$ for $B$ in the sense of Alexandrov and the $\mathcal{F}K$-concavity of $f$. More precisely, for the definition of the non-expanding map $\Psi$ the curvature bound is essential 
because it guarantees the existence of quasi-geodesics, tangent $K$-cones $C_p^K$ at $p\in\partial B$ and non-expanding exponential maps $\exp_p$. 
Then the contraction property of $\Psi$ comes from the $\mathcal{F}K$-concavity of $f$ and the non-expanding property of $\exp_p$.
And also the proof of Proposition \ref{singularity} relies on $\Psi$ and its properties (\cite{albi}).
So it is convenient to assume that $B$ is an Alexandrov space with CBB by $K$ and $f$ is $\mathcal{F}K$-concave. 
\end{remark}
We have the Conjecture \ref{conj} but at the moment we are not able to proof it. But one could ask if it is true when $F$ is a warped product itself 
and satisfies a curvature-dimension bound in the sense of our main theorem. 
In this situation $F$ would not be a manifold and singularities would occur. However the proof of the following corollary shows that an iterated warped product is essentially again a simple warped product.
\begin{corollary}Let $B_2$ be complete, $d_2$-dimensional
space with CBB by $K_2$ such that $B_2\backslash\partial B_2$ is a Riemannian manifold
and let $f_2:B_2\rightarrow\mathbb{R}_{\geq0}$ be $\mathcal{F}K_2$-concave and smooth on $B_2\backslash\partial B_2$. 
Assume $\emptyset\neq\partial B_2\subseteq f_2^{-1}(\left\{0\right\})$.  Let $B_1$ be complete, $d_1$-dimensional
Riemannian manifold with CBB by $K_1$ where $K_1\geq 0$ such that
\begin{itemize}
\item[] $|\nabla f_2|_p\leq\sqrt{K_1}$ for all $p\in \partial B_2$.
\end{itemize}
Let $f_1:B_1\rightarrow\mathbb{R}_{\geq0}$ be a smooth and $\mathcal{F}K_1$-concave.
Assume $\emptyset\neq\partial B_1\subseteq f^{-1}_1(\left\{0\right\})$.
Let $(F,m_F)$ be a weighted, complete Finsler manifold. Let $N\geq 1$ and $K_F\in\mathbb{R}$. 
If $N=1$ and $K_F>0$, we assume that $\diam F\leq \pi/\sqrt{K_F}$. In any case $F$ satisfies $CD((N-1)K_F,N)$ where $K_F\geq 0$ such that
\begin{itemize}
\item[] $|\nabla f_1|_p\leq\sqrt{K_F}$ for all $p\in \partial B_1$.
\end{itemize}
Then the $N+d_1$-warped product $B_2\times^{N+d_1}_{f_2}\big(B_1\times^N_{f_1} F\big)$ satisfies $CD((N+d_1+d_2-1)K_2,N+d_1+d_2)$.
\end{corollary}
\begin{proof}
First we see that 
\begin{align*}
B_2\times^{N+d_1}_{f_2}\big(B_1\times^N_{f_1} F\big)=\big(B_2\times^{d_1}_{f_2} B_1\big)\times^N_{f_2f_1} F.
\end{align*} as metric measure spaces.
This comes from the fact that the warped product measure in both cases is
\begin{align*}
f_2^{N+d_1}d\vol_{B_2}\otimes \big(f_1^Nd\vol_{B_1}\otimes dm_F\big)=(f_1f_2)^N\big(f_2^{d_1}d\vol_{B_2}\otimes d\vol_{B_1}\big)\otimes dm_F
\end{align*}
and the warped product metrics coincide because in both cases the length structure is given by
\begin{align*}
\mbox{L}(\gamma)=\int_0^1\sqrt{|\dot{\alpha_2}(t)|^2+f_2^2\circ\alpha_2(t)|\dot{\alpha}(t)_1|^2+f_2^2\circ\alpha_2(t)f_1^2\circ\alpha_1(t)\mathcal{F}_F^2(\dot{\beta}(t))}dt.
\end{align*}
Hence it is enough to check that $(B_2\times^{d_1}_{f_2} B_1)\times^N_{f_2f_1} F$ satisfies the required curvature-dimension bound. 

We know by Theorem \ref{albi} that $B_2\times^{}_{\scriptscriptstyle{f_2}} B_1=:B$ is a space with CBB by $K_2$. It is easy to see that its boundary is $B_2\times^{}_{\scriptscriptstyle{f_2}} \partial B_1=\partial B$ 
and that the singularity points $\partial B_2$ are a subset of $\partial B$. It follows that $B\backslash\partial B$ is a Riemannian manifold.
Then we know that if $(p_2,p_1)\in\partial B$, we have $f_2(p_2)f_1(p_1)=0$ and so $\partial B\subset f^{-1}(\left\{0\right\})$ where $f=f_2f_1$. Then we can calculate that $f$ is $\mathcal{F}K_2$-concave 
and that it satisfies 
\begin{align*}
|\nabla f|_{(p_2,p_1)}\leq\sqrt{K_F}\hspace{5pt}\mbox{ for all }(p_2,p_1)\in \partial B,
\end{align*}
where the modulus of the gradient is taken with respect to the warped product metric of $B_2\times^{d_1}_{f_2} B_1$. 
Thus the assumptions of Theorem \ref{MainTheorem} are fulfilled and the result follows.
\end{proof}

\section{Appendix}

\paragraph{Existence of optimal maps}
\begin{proposition}\label{geod}
For any $(p,x)\in\hat{C}$ and any $\xi+v\in T\hat{C}_{(p,x)}$ there is a unique geodesic $\gamma$ starting in $(p,x)$ with 
initial tangent vector $\dot{\gamma}(0)=\xi+v$. 
\end{proposition}
\begin{proof} If $v=0$, $\gamma(t)=(\alpha(t),\beta(0))$ is a geodesic in $B$ 
and hence uniquely determined by $\dot{\alpha}(0)$. Otherwise we have $\mathcal{F}_F^2(\dot{\beta})f^4(\alpha)=const=:c$ (see Theorem \ref{fundamentaltheorem}) and $\alpha$ is determined by 
$$
\nabla_{\dot{\alpha}}\dot{\alpha}=-\nabla\textstyle{\frac{c}{2f^2}}|_{\alpha}
$$
and $\alpha(0)$ and $\dot{\alpha}(0)$. Together with the uniqueness property of geodesics in $F$, the statement follows.
\end{proof}
For the rest of this section $c$ stands for the cost function $c((p,x),(q,y))=\textstyle{\frac{1}{2}}|(p,x),(q,y)|^2=\textstyle{\frac{1}{2}}\inf\mathcal{A}(\gamma)$ where the infimum is taken over 
Lipschitz curves connecting $(p,x)$ and $(q,y)$.
We need some background information on $c$-concave functions where we also refer to \cite{mccann}.
\begin{definition}
Let $X,Y\subset \hat{C}$ be compact. Given an arbitrary function $\phi:X\rightarrow\mathbb{R}\cup \left\{-\infty\right\}$, its $c$-transform $\phi^c:Y\rightarrow\mathbb{R}\cup\left\{-\infty\right\}$ relative to $(X,Y)$ 
is defined by $$\phi^c((q,y)):=\inf_{(p,x)\in X}\left\{c((p,x),(q,y))-\phi((p,x))\right\}.$$
Similar we define the $c$-transform of a function $\psi:Y\rightarrow \mathbb{R}\cup\left\{-\infty\right\}$ relative to $(Y,X)$.
A function $\phi:X\rightarrow\mathbb{R}\cup\left\{-\infty\right\}$ is said to be $c$-concave relative to $(X,Y)$ if it is not 
identical $-\infty$ and if there is a function $\psi:Y\rightarrow\mathbb{R}\cup\left\{\infty\right\}$ such
that $\psi^c=\phi$.
\end{definition}
\begin{lemma}
If $\phi$ is $c$-concave relative to $(X,Y)$, then it is Lipschitz continuous with respect to $|\cdot,\cdot|$ and its Lipschitz constant is bounded above by some constant depending only on $X$ and $Y$.
\end{lemma}
\begin{remark}\label{mb}
Since a $c$-concave function is Lipschitz continuous, it is differentiable almost everywhere. 
We also have that $d\phi: \hat{C}\rightarrow T^*\hat{C}$ is measurable (see \cite[Lemma 4]{mccann}).
\end{remark}
\begin{definition}
Let $M$ be a manifold and $f:M\rightarrow \mathbb{R}$ be a function. A co-vector $\alpha\in T^*M_x$ is called subgradient of $f$ at $x$ if we have
\begin{equation*}
f(\sigma(1))\geq f(\sigma(0))+\alpha(\dot{\sigma}(0))+o(\mathcal{F}(\dot{\sigma}(0))
\end{equation*}
for any geodesic $\sigma:[0,1]\rightarrow M$ with $\sigma(0)=x$. The set of subgradients at $x$ is denoted by $\partial_-^*f(x)$. Analogously we can define the set $\partial_+^*f(x)$ of supergradients at $x$.
\end{definition}
\begin{remark}
If $f$ admits a sub- and supergradient at $x$, it is differentiable at $x$ and $\partial_-^*f(x)=\partial_+^*f(x)=\left\{df_x\right\}$ (\cite[Proposition 10.7]{viltot}).
\end{remark}
\begin{proposition}
Suppose $\gamma:[0,1]\rightarrow \hat{C}$ is a constant speed geodesic joining $(p,x)$ and $(q,y)$. 
Then $f(\cdot)=c(\cdot,(q,y))$ has supergradient $-d_v\mathcal{F}^2_{\hat{C}}|_{\dot{\gamma}(0)}\in T^*\hat{C}_{\gamma(0)}$ at $(p,x)$ where
\begin{align*}
d_v\mathcal{F}^2_{\dot{\gamma}(0)}(w)=\frac{d}{dt}\mathcal{F}^2(\dot{\gamma}(0)+tw)\hspace{5pt}\mbox{ for }w\in T\hat{C}|_{\gamma(0)}
\end{align*}
\end{proposition}
\begin{proof}
Let $(\tilde{p},\tilde{x})$ and $(\tilde{q},\tilde{y})$ are points that are very close to $(p,x)$ and $(q,y)$ such that there are unique geodesics $\sigma_0, \sigma_1:[0,1]\rightarrow \hat{C}$
between $(p,x)$ and $(\tilde{p},\tilde{x})$ and between $(q,y)$ and $(\tilde{q},\tilde{y})$, respectively. Let $\tilde{\gamma}$ be an arbitrary curve that connects $(\tilde{p},\tilde{x})$ 
and $(\tilde{q},\tilde{y})$. Then we have by the formula of first variation
\begin{align*}
\int_0^1\mathcal{F}^2_C(\dot{\tilde{\gamma}}(t))dt&=\int_0^1\mathcal{F}^2_C(\dot{\gamma}(t))dt+d_v\mathcal{F}^2_C|_{\dot{\gamma}(1)}(\dot{\sigma}_1(0))-d_v\mathcal{F}^2_C|_{\dot{\gamma}(0)}(\dot{\sigma}_0(0))+o(\sup_{\scriptscriptstyle{t\in[0,1]}}|\gamma(t),\tilde{\gamma}(t)|).
\end{align*}
Hence we can proof for some $\tilde{\gamma}$ with $(\tilde{q},\tilde{y})=(q,y)$ that
\begin{align*}
c((\tilde{p},\tilde{x}),(q,y))\leq \int_0^1\mathcal{F}^2_C(\dot{\tilde{\gamma}}(t))dt\leq c((p,x),(q,y))-d_v\mathcal{F}^2_C|_{\dot{\gamma}(0)}(\dot{\sigma}_0(0))+o(|(p,x),(\tilde{p},\tilde{x})|)
\end{align*}
which means that $c(\cdot,(q,y))$ has supergradient $-d_v\mathcal{F}^2_C|_{\dot{\gamma}(0)}$. For more details we refer to \cite[Proposition 10.15]{viltot}.
\end{proof}
\begin{lemma}\label{cconcave}
Let $X,Y\subset \hat{C}$ be two compact subsets and $\phi:X\rightarrow \mathbb{R}$ be a $c$-concave function.
If $\phi$ is differentiable in $(p,x)\in X$, 
and 
\begin{align}\label{cc}c((p,x),(q,y))=\phi((p,x))+\phi^c((q,y)),\end{align}
then there is a geodesic $\gamma=\gamma^{(p,x)}$ between $(p,x)$ and $(q,y)$ satisfying $-d\phi_{(p,x)}(\dot{\gamma}(0))=\mathcal{F}^2(\dot{\gamma}(0))$. The point $(q,y)$ and the geodesic $\gamma$ are uniquely determined by $(p,x)$ and $\phi$.
%
%
%
%
\end{lemma}
\begin{proof}
By definition of $c$-concave functions we have $\geq$ in (\ref{cc}) for any pair of points. Now choose $(p,x)$ and $(q,y)$ such that (\ref{cc}) holds and $\phi$ is differentiable at $(p,x)$. 
Then we have for any $(\tilde{p},\tilde{x})$
\begin{align*}
\phi((\tilde{p},\tilde{x}))-\phi((p,x))\leq c((\tilde{p},\tilde{x},(q,y))-c((p,x),(q,y))
\end{align*}
Instead of the point $(\tilde{p},\tilde{x})$ we insert a curve $\sigma:(0,\epsilon)\rightarrow X$ (parametrized by arclength). Then we deduce 
\begin{align*}
d\phi_{(p,x)}(\dot{\sigma})=\frac{d}{d\epsilon}\phi\circ\sigma|_{\epsilon=0}\leq \liminf_{\epsilon\rightarrow 0}\frac{c(\sigma(\epsilon),(q,y))-c((p,x),(q,y))}{\epsilon}
\end{align*}
It follows that $d\phi_{(p,x)}$ is a subgradient of $c(\cdot,(q,y))$ at $(p,x)$. But by the previous proposition $c(\cdot,(q,y))$ has also a supergradient at $(p,x)$. Thus it is differentiable at $(p,x)$ with
\begin{align}\label{dphi}
 -d_v\mathcal{F}^2_{\hat{C}}|_{\dot{\gamma}(0)}=dc(\cdot,(q,y))_{(p,x)}=d\phi_{(p,x)}
\end{align}
where $\gamma$ is some geodesic that connects $(p,x)$ and $(q,y)$. Now we know that $\mathcal{F}^2_{C}$ is strictly convex in $v$ and $C^1$. Thus the co-vector $d_v\mathcal{F}^2_{\hat{C}}|_{\dot{\gamma}(0)}$ determines
$\dot{\gamma}(0)$ uniquely by $d_v\mathcal{F}^2_{\hat{C}}|_{\dot{\gamma}(0)}(w)=\mathcal{F}^2(w)$ and therefore $\gamma$ by Proposition \ref{geod}. So (\ref{dphi}) and the strict convexity of $\mathcal{F}^2_C$ 
with respect to $v$ determines $y$ uniquely.
\end{proof}
\begin{remark}
On $T\hat{B}\oplus TF\backslash 0_F$ we have that $\dot{\gamma}^{(p,x)}$ coincides with the gradient of $-\phi$ at $(p,x)$, that can be defined via Legendre transformation, and $\gamma^{(p,x)}(t)=\exp(-t\nabla\phi_{(p,x)})$.
On $T\hat{B}\oplus 0_F$ it coincides with the gradient that comes from the Riemannian structure on $B$.
 The map $(p,x)\mapsto \dot{\gamma}^{(p,x)}$ is measurable
because $\phi$ is Lipschitz (see Remark \ref{mb}) and the transformation $\alpha\in T^*C_{\scriptscriptstyle{(x,p)}}\mapsto \alpha^*\in TC_{\scriptscriptstyle{(p,x)}}$ is continuous 
where $\alpha^*$ is uniquely determined by $\alpha(\alpha^*)=\mathcal{F}^2(\alpha^*)$. 
Now one can deduce that also $(p,x)\mapsto\gamma^{\scriptscriptstyle{(p,x)}}(1)$ is measurable by considering the ``exponential map'' separately on $T\hat{B}\oplus TF\backslash 0_F$ and $T\hat{B}\oplus 0_F$.
In \cite[Theorem 10.28]{viltot} measurability of $T$ is deduced by applying a measurable selection theorem. 
\end{remark}
\begin{proof}[Proof of Theorem \ref{optimalmap}]
Let $\pi$ be an optimal transference plan. By Kantorovich duality there exists a $c$-concave function $\phi$ such that $\phi((p,x))+\phi^c((q,y))\leq c((p,x),(q,y))$ everywhere on $\supp\pi\subset X\times Y$, 
with equality $\pi$-almost surely. Since $\phi$ is differentiable $\m_C$-almost surely and since $\mu$ is absolutely continuous with respect to $\m_C$, we can define $T$ by Lemma \ref{cconcave}
$\m_C$-almost surely by $T((p,x))=\gamma^{(p,x)}(1)$ where $\gamma^{(p,x)}$ is uniquely given by $-d\phi_{(p,x)}(\dot{\gamma}(0))=\mathcal{F}^2(\dot{\gamma}(0))$. 
Thus $\pi$ is concentrated on the graph of $T$, or equivalently $\pi=(Id_{\hat{C}}, T)_*\mu$. Now from \cite[Lemma 4.9]{ohtafinsler1} 
we know that in our setting $-d\phi$ is unique among all maximizers $(\phi,\phi^c)$ of Kantorovich duality as long as the initial measure is absolutely continuous. So also $T$ and $\pi$ are unique.
\end{proof}

\bibliography{new}
\end{document}